\documentclass[a4paper]{article}

\usepackage{pgf}
\usepackage{mathtools}
\usepackage{graphicx}
\usepackage{amsthm}
\usepackage{amsmath}
\usepackage{amsfonts}
\usepackage{MnSymbol}
\usepackage[utf8]{inputenc}
\usepackage[english]{babel}
\usepackage{cite}
\usepackage[toc,page]{appendix}
\usepackage{url}
\usepackage{youngtab}
\usepackage{ytableau}
\usepackage{comment}
\usepackage[margin=1.2in]{geometry}
\usepackage{floatrow}
\usepackage{relsize}
\usepackage{float}

\usepackage[hidelinks]{hyperref}

\usepackage{tikz}
\usetikzlibrary{arrows}
\usepackage{tkz-graph}

\newtheorem{thm}{Theorem}

\newtheorem{corollary}[thm]{Corollary}
\newtheorem{lemma}[thm]{Lemma}
\newtheorem{defn}[thm]{Definition}

\newtheorem{example}[thm]{Example}

\newtheorem{rem}[thm]{Remark}

\newcommand{\E}{\mathrm{E}}
\newcommand{\Var}{\mathrm{Var}}
\newcommand{\sh}{\Phi}

\newcommand{\pro}{\kappa}

\newcommand{\iso}{\pi}

\newcommand{\best}{T_\lambda^{\rightarrow}}
\newcommand{\worst}{T_\lambda^{\downarrow}}
\newcommand{\SEQ}{T_\lambda^{\rightarrow}}
\newcommand{\OSEQ}{T_\lambda^{\downarrow}}
\newcommand{\RL}{P}
\newcommand{\LR}{P}
\newcommand{\SYT}{\textnormal{SYT}}
\newcommand{\eig}{\textnormal{eig}}
\newcommand{\eval}{\textnormal{eval}}
\newcommand{\BLR}{P}

\newcommand{\righthand}{R^{i}}
\newcommand{\lefthand}{L^{i}}
\newcommand{\NA}{N_{\alpha}(n)}
\newcommand{\NW}{N_{w}(n)}

\DeclareSymbolFont{bbold}{U}{bbold}{m}{n}
\DeclareSymbolFontAlphabet{\mathbbold}{bbold}
\newcommand{\ind}{\mathbbold{1}}

\begin{document}
	
\title{Cutoff for a One-sided Transposition Shuffle}
\author{Michael E. Bate, Stephen B. Connor, and Oliver Matheau-Raven\thanks{PhD funded by the EPSRC grant: EP/N509802/1}\\University of York, UK}
\date{\today}

\maketitle
\begin{abstract}
We introduce a new type of card shuffle called a \emph{one-sided 
transposition shuffle}. At each step a card is chosen uniformly from 
the pack and then transposed with another card chosen uniformly from \emph{below} it. This defines a random walk on the symmetric group generated by a distribution which is non-constant on the conjugacy class of transpositions. Nevertheless, we provide an explicit formula for all eigenvalues of the shuffle by demonstrating a useful correspondence between eigenvalues and standard Young tableaux. This allows us to prove the existence of a total-variation cutoff for the 
one-sided transposition shuffle at time $n\log n$. We also study a weighted 
	generalisation of the shuffle which, in particular, allows us to recover the well known mixing time of the classical random transposition shuffle. 
\end{abstract}

 \begin{quotation}
	\noindent
	Keywords and phrases:\\
	\textsc{Mixing time; Cutoff phenomenon; Coupon collecting; Representation theory; Young tableaux}
\end{quotation}

 \begin{quotation}\noindent
\end{quotation}

\section{Introduction}

Consider a stacked deck of $n$ distinct cards, whose positions are labelled by elements of the set $[n]:=\{1,\dots,n\}$ from bottom to top. Any shuffle which involves choosing two positions and switching the cards found there is 
called a transposition shuffle, and may be viewed as a random walk on the symmetric group $S_n$. (If the two positions coincide then no cards are moved.) Diaconis and Shahshahani \cite{diaconis1981generating} were the first to study transposition shuffles using the representation theory of $S_{n}$; they famously showed that the \emph{random transposition shuffle}, in which the two positions are chosen independently and uniformly on $[n]$, takes $(n/2)\log n$ steps to randomise the order of the deck. (This is known as the \emph{mixing time} of the shuffle.) The use of representation theory, both in \cite{diaconis1981generating} and successive works (e.g. \cite{berestycki2011, berestycki2019cutoff}), relies heavily on the fact that the distribution generating the shuffle is constant on conjugacy classes of $S_n$.

Since then a variety of algebraic and probabilistic techniques have been employed to study different types of transposition shuffle. Notable examples include the ``transpose top and random'' shuffle \cite{saloff2004random}, the ``adjacent transpositions'' shuffle \cite{Lacoin2016}, and a generalisation of the  latter in which the two cards are constrained to lie within a certain (cyclical) distance of one another \cite{berestycki2012}. All of these shuffles have the property that, at each step, the transposition to be applied is chosen uniformly from within some set which generates the entire group $S_n$.

\emph{Semi-random transposition shuffles} form an interesting class of Markov chains (see e.g. \cite{Mossel2004, Pymar2011, pinsky2012cyclic}). In this class the right hand chooses a card uniformly at random, the left hand chooses a card via some (possibly time-inhomogeneous) independent stochastic process, and then the two chosen cards are transposed. A universal upper bound of $O(n \log n)$ on the mixing time of any semi-random transposition shuffle was established by Mossel, Peres, and Sinclair \cite{Mossel2004}.

In the present work we introduce a new class of shuffles called \emph{one-sided transposition shuffles}: these have the defining property that at step $i$ the right hand's position ($\righthand$) is chosen according to some arbitrary distribution on $[n]$, and given the value of $\righthand$ the distribution of the left hand's position ($\lefthand$) is supported on the set $\{1,\dots,\righthand\}$. For the majority of this paper we shall focus on the case when the right and left hands are both chosen uniformly from their possible ranges, but in Section~\ref{GENERALISATIONS} we shall extend our main results to the case when $\righthand$ is chosen using a particular type of weighted distribution.

Our setup differs significantly from the previously studied shuffles mentioned above. The dependence between the left and right hands means that this does not fall into the class of semi-random transpositions. Furthermore, although the generating set for our shuffle is the entire conjugacy class of transpositions, the distribution that we impose upon this set is in general far from uniform. (E.g. when right and left hands are both uniform on their permitted ranges, the probabilities attached to different transpositions range from $1/n^2$ to $1/2n$; see Definition~\ref{LRprob} below.) We note that there can clearly be no universal upper bound on the mixing time of these shuffles without imposing further constraints on the distribution of the left hand, since the shuffle can be slowed arbitrarily by increasing the probability that the two hands choose the same position.

\medskip
In order to state our results we briefly introduce some notation and 
terminology. 

\begin{defn} 
	\label{LRprob}
	The (unbiased) one-sided transposition shuffle $\LR_{n}$ is the ergodic random walk on 
	$S_n$ generated by the following distribution on the conjugacy 
	class of transpositions:
	\[\RL_{n}(\tau) = 
	\begin{cases}
	\frac{1}{n}\cdot\frac{1}{j} &  \text{if } \tau = (i\,j) \text{ for some } 
	1\leq i 
	\leq j 
	\leq n\\
	0 & \text{otherwise.}
	\end{cases}
	\]
	
	We use the convention that all `transpositions' $(i \,i)$ are equal to the identity element $id$, and therefore 
	$\RL_{n}(id) = \frac{1}{n} (1 + \frac{1}{2} + \dots + 
	\frac{1}{n}) = 
	\frac{H_{n}}{n}$, where $H_{k}$ denotes the $k^{\textnormal{th}}$ harmonic 
	number.	We write $P_n^t$ for the $t-$fold convolution of $P_n$ with itself. 
\end{defn}
This shuffle is clearly both reversible and transitive, and has stationary distribution equal to the uniform distribution on $S_{n}$, denoted $\pi_{n}$; that is, $P_n^t(\sigma)\to 1/n!$ as $t\to\infty$ for all $\sigma\in S_n$. In order to study the rate of this convergence, we begin by recalling that the total variation distance between two probability distributions $P$ and $Q$ on $S_{n}$ is defined by 
	\[\lVert P-Q \rVert_{\textnormal{TV}} = \sup_{A\subseteq S_{n}} 
	|P(A) -Q(A)| = \frac{1}{2}\sum_{\sigma \in S_{n}} |P(\sigma) - Q(\sigma)| 
	.\]
Now consider a sequence of distributions $\{Q_{n}\}_{n\in\mathbb{N}}$ on state 
spaces $\{S_n\}_{n\in\mathbb N}$ with corresponding stationary distributions 
$\{\pi_{n}\}_{n\in\mathbb{N}}$. 
We may define the total variation mixing time $t_n^{\textnormal{mix}}(\varepsilon)$ for the distribution $Q_{n}$ as follows:
\[t_n^{\textnormal{mix}}(\varepsilon) = \min\{t : \lVert Q_{n}^{t} - \pi_{n} \rVert_{\textnormal{TV}} < \varepsilon\} .\]
%By convention we set $t_n^{\textnormal{mix}}:= t_n^{\textnormal{mix}}(1/4)$.
It is well known that many natural sequences of 
this kind exhibit behaviour known as a \emph{cutoff phenomenon}, whereby the 
convergence to equilibrium occurs more and more sharply as $n\to\infty$.
\begin{defn}
	A sequence of distributions $\{Q_{n}\}$ exhibits a (total variation) \emph{cutoff} at time 
	$\{t_{n}\}$ with a \emph{window} of size $\{w_{n}\}$ if $w_{n}=o(t_{n})$ and the 
	following limits hold:
	\begin{eqnarray*}
	\lim_{c\rightarrow \infty}\limsup_{n\rightarrow \infty} \lVert Q_{n}^{t_{n} 
		+cw_{n}} -\pi_n \rVert_{\textnormal{TV}}&=&0 \\
		\lim_{c\rightarrow \infty}\liminf_{n\rightarrow \infty} \lVert Q_{n}^{t_{n} 
		-cw_{n}} -\pi_n \rVert_{\textnormal{TV}}&=&1\,.
	\end{eqnarray*}
\end{defn}
Existence of a cutoff implies that $t_n^{\textnormal{mix}}(\varepsilon) \sim t_n$ for all $\varepsilon\in(0,1)$. The main conclusion of this work is that the one-sided transposition shuffle exhibits a cutoff at time $t_n = n\log n$.
\begin{thm}
	\label{LR main theorem}
	The one-sided transposition shuffle $\RL_{n}$  exhibits a cutoff at time $n\log 
	n$. Specifically, for any $c_{1}>0$ and $c_{2}>2$
	\begin{eqnarray}
	\limsup_{n\rightarrow \infty} \,\lVert \RL_{n}^{n\log n+c_{1}n} - \pi_{n} \rVert_{\textnormal{TV}} & \leq & 
	\sqrt 2e^{-c_{1}} \,, \label{eqn:main_UB}\\
	\text{and} \quad \liminf_{n\to\infty}\,\lVert \RL_{n}^{n\log n - n\log\log n - c_{2}n} - 
	\pi_{n}
	\rVert_{\textnormal{TV}} 
	& \geq & 1 - \frac{\pi^{2}}{6(c_{2}-2)^{2}}\,.\label{eqn:main_LB}
	\end{eqnarray}
\end{thm}

The lower bound on the mixing time in \eqref{eqn:main_LB} will be obtained via a coupling argument which allows us to compare the one-sided transposition shuffle to a variation of a coupon-collecting problem. To establish the upper bound for the cutoff in Theorem \ref{LR main 
	theorem} we shall make use of a classical $\ell^{2}$ bound on total variation distance.
\begin{lemma}[Lemma 12.16 \cite{Levin2017}]
	\label{ClassicL2}
	Let $Q$ be the transition matrix for a reversible, transitive, aperiodic Markov chain on finite state space $\mathcal{X}$, 
	with stationary distribution $\pi$. Let the eigenvalues of $Q$ be denoted 
	$\beta_{i}$, with $1=\beta_{1} > \beta_{2} \geq \dots \geq \beta_{|\mathcal{X}|} >
	-1$. Then
	\[4\lVert Q^{t} - \pi \rVert^{2}_{\textnormal{TV}} \leq \sum_{i\neq 1} 
	\beta_{i}^{2t} .\]
\end{lemma}

The spectrum of the one-sided 
transposition shuffle will be analysed using the 
representation theory of the symmetric group, and our method exhibits some 
algebraic 
features which are of independent interest. The key result is describing an explicit method for obtaining the eigenvectors of $\LR_{n+1}$ 
from 
those of $\LR_n$; this is called \emph{lifting} eigenvectors. It is interesting that the technique of lifting eigenvectors allows the use of representation theory to analyse shuffles which are not constant on conjugacy classes.  
In order to work, this technique requires a particularly close relationship between the shuffle on $n$ and $n+1$ cards -- in our examples of one-sided transposition shuffles on $n+1$ cards, the $(n+1)$-th card only moves when $\righthand =n+1$, and for all other choices of $\righthand$ the shuffle on $n+1$ cards behaves like the shuffle on $n$ cards (see equation \eqref{eqn:lrdifference} for more details). 
An analysis involving the lifting of eigenvectors was first used in the recent work of Dieker and Saliola \cite{dieker2018spectral} which studied the eigenspaces of the  \emph{random-to-random shuffle}; this analysis was used by Bernstein and Nestoridi~\cite{bernstein2019} to prove the existence of a cutoff for this shuffle at time $(3/4)n\log n$.
Lafreni{\`e}re \cite{Lafreni2018} subsequently showed that similar lifting techniques could be applied to more general ``symmetrized shuffling operators''.
In this paper we make several non-trivial changes to the technique developed by Dieker and 
Saliola in order to employ it in the analysis of the one-sided transposition shuffle: we believe that this is the first time such a technique has been shown to be applicable to either a non-symmetrized shuffle or a transposition shuffle. 
We suspect that with suitable modifications the technique of lifting eigenvectors may be used to analyse a whole variety of shuffles for which the standard technique of discrete Fourier transforms fails.

\paragraph{}
Using this method we prove that each eigenvalue of $\LR_n$ corresponds to a 
\emph{standard Young tableau}, and may be computed explicitly from the entries 
in the tableau. That we are able to find such explicit results is remarkable given that the distribution generating the shuffle is not constant on conjugacy classes. Definitions, notation and proofs will be carefully laid out below, but for ease of reference we state here the main two 
results which will underpin our analysis, proofs of which can be found in the Appendix:

\begin{thm}
	\label{Maintheorem}
	The eigenvalues of $\RL_{n}$ are labelled by standard Young tableaux of 
	size $n$, and the eigenvalue represented by a tableau of shape $\lambda$ appears 
	$d_{\lambda}$ times, where $d_\lambda$ is the dimension of $\lambda$.
\end{thm}

\begin{lemma}
	\label{compute}
	For a tableau $T$ of shape $\lambda$ the eigenvalue corresponding to $T$ is given by
	\begin{eqnarray}
	\label{eigenvaluesum}
	\eig(T) = \frac{1}{n}\sum_{\substack{\textnormal{boxes}\\ (i,j)}} 
	\frac{j-i+1}{T(i,j)}\,,
	\end{eqnarray}
	where the sum is performed over all boxes $(i,j)$ in $T$.
\end{lemma}

The organisation of the remainder of the paper is as follows. In 
Section~\ref{UPPERBOUND} we first explore some important properties of 
the eigenvalues for $\LR_{n}$,  and then  use these to prove the upper bound on the mixing time given by Theorem~\ref{LR main theorem}. The 
corresponding lower bound will be proved in Section~\ref{LOWERBOUND}, using 
entirely probabilistic arguments. Finally, in Section~\ref{GENERALISATIONS} we 
will consider a generalisation of the one-sided transposition shuffle, in which 
$\righthand$ is chosen according to a non-uniform distribution: we show that the 
algebraic technique developed for $\RL_{n}$ holds in this more general setting, and that 
the well-known mixing time for the (standard) random transposition shuffle may 
be recovered in this way.

\section{Upper Bound}
\label{UPPERBOUND}
\subsection{Eigenvalue Analysis}
\label{EIGENVALUES}

Before we establish relations of the eigenvalues for $\LR_{n}$ we first will 
recall some standard definitions about partitions, Young diagrams and Young 
tableaux \cite{james_1984}. 
A \emph{partition} of $n$ 
is a tuple of positive integers $\lambda = (\lambda_{1}, \dots, \lambda_{r})$ 
such that, $\sum \lambda_{i} =n$, and $\lambda_{1} \geq \dots \geq \lambda_{r}$.
We write $\lambda \vdash n$. 
We call $r$ the \emph{length} of $\lambda$ and denote it $l(\lambda)$. We denote the 
partition $(1,\dots,1) \vdash n$ as $(1^{n})$. 

Every partition $\lambda$ has an 
associated 
\emph{Young diagram}, 
made by forming a left adjusted stack of boxes with rows labelled downwards and 
row $i$ having 
$\lambda_{i}$ boxes.  We often blur the distinction between a partition and 
its Young diagram, e.g. $(4,2) = 	
\ytableausetup{mathmode,baseline,aligntableaux=center,boxsize=0.6em}
\ydiagram{4,2}$. 
We may refer to the boxes of a diagram $\lambda$ by using 
coordinates $(i,j)$ to mean the box in the $i^{th}$ row and $j^{th}$ column. 
Given a partition $\lambda\vdash n$, we may form the 
\emph{transpose} of $\lambda$, denoted $\lambda^{\prime}$, by turning rows into 
columns in the Young diagram. We have $\lambda \vdash n$ if and only if $\lambda^{\prime} \vdash n$.
 
We have a partial order on partitions of $n$ called the \emph{dominance order}: 
in terms of Young diagrams, for two partitions $\mu,\lambda \vdash n$, we write 
$\lambda \trianglerighteq \mu$ if we can form $\mu$ by moving boxes of $\lambda$
down and to the left. 
We have $\lambda \trianglerighteq \mu$ if and only if $\mu'  \trianglerighteq \lambda'$ \cite[Lemma 1.4.11]{james_1984}.

Given two partitions $\mu,\lambda$ of different sizes, we 
say $\mu \subset \lambda$ if $\mu$ is fully contained in $\lambda$ when we 
align the Young diagrams of $\mu$ and $\lambda$ at the top left corners;
equivalently, if we write $\lambda = (\lambda_1,\dots,\lambda_r)$ and $\mu=(\mu_1,\dots,\mu_s)$, 
this means that $s\leq r$ and $\mu_i\leq \lambda_i$ for each $1\leq i \leq s$.

Given a partition $\lambda\vdash n$, we may form a \emph{Young tableau} $T$ by 
putting numbers into the boxes of (the Young diagram of) $\lambda$. 
A \emph{standard Young tableau} $T$ is one in which the numbers 
$1,\dots, n$ occur once each and such that the values 
are increasing across rows and down columns. 
The set of standard Young tableaux of shape $\lambda$  is 
denoted by $\SYT(\lambda)$. The size of the set $\SYT(\lambda)$ is
called the \emph{dimension} of $\lambda$, denoted $d_{\lambda}$. 
For a tableau $T$, form the transpose of $T$ denoted $T^{\prime}$ by turning rows 
into columns and preserving the value in each box. If $T \in \SYT(\lambda)$, 
then $T^{\prime} \in \SYT(\lambda^{\prime})$.
Given a tableau $T$,  let $T(i,j)$ be the value of box $(i,j)$ in $T$ if it 
exists and undefined otherwise. 

For any $\lambda \vdash n$, define the tableau 
$\best$ by inserting the numbers $1,\dots, n$ from left to right. Define the 
tableau $\worst$ by inserting the numbers $1,\dots, n$ from top to bottom.
From the introduction we know the eigenvalues for $\LR_{n}$ are labelled by 
Young tableaux of size $n$, and Lemma \ref{compute} gives an explicit 
formula for the eigenvalue associated to a given tableau. 
Before applying the bound in Lemma \ref{ClassicL2}, we first 
investigate relationships between the eigenvalues. 
We show that the eigenvalue corresponding to $T \in \SYT(\lambda)$ is bounded by 
the eigenvalues for $\best$ and $\worst$.  
To simplify our upper bound calculation, 
we prove that we only need to consider the partitions for which $\best$ gives a positive eigenvalue. Lastly, we prove that the eigenvalues corresponding to $\best,\worst$ decrease 
as one moves down the dominance order of partitions. 
We first illustrate some of the preceding definitions and discussion with an example.

\begin{example}
Let $\lambda=(3,2)\vdash 5$.
Then $\SYT(\lambda)$ has five elements, so $d_\lambda = 5$.
These $5$ tableaux, together with the associated eigenvalues calculated using Lemma \ref{compute},
are given in Table~\ref{eigenvaluetable} below.  
In this table, $\best$ is the first tableau listed and $\worst$ is the last one; we can see that the corresponding eigenvalues bound all the others.

\begin{table}[H]
	\centering
	{
		\renewcommand{\arraystretch}{1.5}
	\begin{tabular}{c|c|c|c|c|c}
		$T \in \SYT((3,2))$ &
		$\ytableausetup{mathmode,baseline,aligntableaux=bottom,boxsize=1em}
		\begin{ytableau} 1 & 2 & 3\\ 4 & 5 
		\end{ytableau}$ & 
		$\ytableausetup{mathmode,baseline,aligntableaux=bottom,boxsize=1em}
		\begin{ytableau} 1 & 2 & 4\\ 3 & 5
		\end{ytableau} $ & 
		$\ytableausetup{mathmode,baseline,aligntableaux=bottom,boxsize=1em}
		\begin{ytableau} 1 & 2 & 5\\ 3 & 4
		\end{ytableau} $ & 
		$\ytableausetup{mathmode,baseline,aligntableaux=bottom,boxsize=1em}
		\begin{ytableau} 1 & 3 & 4\\ 2 & 5
		\end{ytableau} $ & 
		$\ytableausetup{mathmode,baseline,aligntableaux=bottom,boxsize=1em}
		\begin{ytableau} 1 & 3 & 5\\ 2 & 4
		\end{ytableau}$   \\\hline
		$\eig(T)$ & $0.64$	& $0.59$ & $0.57$ &  $0.52\overline{3}$	& $0.50\overline{3}$
%		\multicolumn{1}{c|}{$\eig(T)$} &\multicolumn{1}{|c|}{$0.64$} 
%		&\multicolumn{1}{|c|}{$0.59$} & 
%		\multicolumn{1}{|c|}{$0.57$} & \multicolumn{1}{|c|}{$0.52\overline{3}$} 
%		& \multicolumn{1}{|c}{$0.50\overline{3}$}
	\end{tabular}
}
	\caption{Eigenvalues corresponding to $T \in \SYT((3,2))$.}
	\label{eigenvaluetable}
\end{table}
\end{example}

We can now begin our analysis of the eigenvalues by showing how swapping numbers in a tableau affects the corresponding eigenvalue.

\begin{lemma}
	\label{swapping}
	Let  $T$ be a general Young tableau. Suppose we form a new tableau $S$ by swapping 
	two values in $T$ which have coordinates $(i_{1},j_{1}),(i_{2},j_{2})$ in 
	$T$. WLOG assume  $T(i_{1},j_{1}) < 
	T(i_{2},j_{2})$. Then the change in corresponding eigenvalues satisfies the following inequality:
	\[\eig(S) - \eig(T) 
	\begin{cases}
	\geq 0 & \text{ if } (i_{1}-i_{2})+ (j_{2}-j_{1}) \geq 0\\
	< 0 & \text{ if }  (i_{1}-i_{2}) + (j_{2}-j_{1})<0\,.
	\end{cases}
	\]	
	Importantly, if we move the larger entry down and to the left the change 
	in eigenvalue is non-negative; if it moves up and to the right then the change is negative.
\end{lemma}

\begin{proof}
	Since $S$ and $T$ agree in all but two entries, Lemma~\ref{compute} tells us that the difference in eigenvalues is given by 
	\begin{align*}
	\eig(S) - \eig(T) &= \frac{1}{n}\left(\frac{j_{1}-i_{1}+1}{T(i_{2},j_{2})} + 
	\frac{j_{2}-i_{2}+1}{T(i_{1},j_{1})}\right) - 
	\frac{1}{n}\left(\frac{j_{1}-i_{1}+1}{T(i_{1},j_{1})} + 
	\frac{j_{2}-i_{2}+1}{T(i_{2},j_{2})}\right) \\
	&= \frac{(i_{1}-i_{2})+(j_{2}-j_{1})}{n} \left(\frac{1}{T(i_{1},j_{1})} - 
	\frac{1}{T(i_{2},j_{2})} 
	\right).\qedhere
	\end{align*}
\end{proof}

This result allows us to prove that the eigenvalues for $\worst$ and $\best$ bound all others for $\SYT(\lambda)$.

\begin{lemma}
	\label{seq}
	Let $\lambda\vdash n$.
	For any $T \in \SYT(\lambda)$ we have the following 
	inequality:
	\begin{eqnarray}
	\eig(\worst) \leq \eig(T) \leq 
	\eig(\best).
	\end{eqnarray}
\end{lemma}

\begin{proof}
	Reading across the rows of $T$, beginning with the first row, identify the first box in which $T$ and $\best$ have different entries; write $(i,j)$ for the coordinates of this box. Due to the way in which $\best$ is constructed, $T(i,j)>\best(i,j)$. Furthermore, the number $T(i,j)-1$ must occur strictly below and to the left of $T(i,j)$, since $T$ is a standard Young tableau. Swapping entries $T(i,j)-1$ and $T(i,j)$ produces a new element of $\SYT(\lambda)$ 
	whose corresponding eigenvalue is no smaller than $\eig(T)$, thanks to Lemma~\ref{swapping}.
	
	We can therefore iterate this procedure, swapping $T(i,j)-1$ with 
	$T(i,j)-2$ etc, until $T(i,j) = \best(i,j)$. Note that at this point the 
	entries in the first $T(i,j)$ boxes of $T$ and $\best$ must agree, moreover 
	these entries are now fixed in place. We now 
	proceed to the next box in which $T$ and $\best$ differ, and repeat: this 
	results in a sequence of swaps which make the entries of $T$ agree with 
	those in $\best$, and which can only ever cause the corresponding 
	eigenvalue to increase. This proves the second inequality in 
	Lemma~\ref{seq}, and the first one follows via an analogous argument on the 
	columns of $T$.
		
\end{proof}

The next result and its corollary establish that when bounding eigenvalues, we only need to consider
those given by $\best$. 

\begin{lemma}
	\label{lem:sumtranspose}
	Let $\lambda\vdash n$.
	For any $T \in \SYT(\lambda)$ we have 
	\[\eig(T) + \eig(T^{\prime}) = \frac{2H_{n}}{n}.\]
\end{lemma}

\begin{proof}
	Let $T\in SYT(\lambda)$. Then
	\begin{align*} \eig(T) +\eig(T')&= \frac 1n \sum_{\substack{\textnormal{boxes}\\ (i,j)\in T}} 
	\frac{j-i+1}{T(i,j)} + \frac 1n \sum_{\substack{\textnormal{boxes}\\ (j,i)\in T'}} 
	\frac{i-j+1}{T'(j,i)} \\
	&= \frac 1n \sum_{\substack{\textnormal{boxes}\\ (i,j)\in T}} 
	\frac{j-i+1}{T(i,j)} + \frac 1n \sum_{\substack{\textnormal{boxes}\\ (i,j)\in T}} 
	\frac{-(j-i)+1}{T(i,j)} = \frac {2H_n}{n}.
	\end{align*}
\end{proof}

\begin{corollary}
	\label{flip}
	Let $\lambda =(\lambda_{1},\dots,\lambda_{r}) \vdash n$, and suppose we 
	have 
	$\eig(\worst) \leq 0$, then we 
	have 
	\begin{eqnarray}
	\label{case2}
	\eig(T_{\lambda^{\prime}}^{\rightarrow}) \geq 
	|\eig(\worst)| \geq  0.
	\end{eqnarray}
\end{corollary}

\begin{proof}
	It follows from Lemma~\ref{lem:sumtranspose} that $\eig(T_{\lambda'}^\rightarrow) + \eig(T_{\lambda}^\downarrow) = 2H_n/n$. Thus if $\eig(\worst) \leq 0$ then 
	\begin{align*}
	\eig(T_{\lambda'}^\rightarrow) = \frac{2H_n}{n} - \eig(T_{\lambda}^\downarrow) \ge - \eig(T_{\lambda}^\downarrow) = |\eig(T_{\lambda}^\downarrow)| \ge 0\,.
	\end{align*}
\end{proof}

We end this section by establishing a relationship between eigenvalues and the dominance ordering on partitions.
\begin{lemma}
	\label{order}
	Let $\lambda, \mu \vdash n$. If $\lambda \trianglerighteq \mu$ then 
 	\begin{eqnarray}
	\eig(\best)& \geq & 
	\eig(T_\mu^{\rightarrow})\\
	\label{down}
	\text{and }\quad \eig(\worst) & \geq &
	\eig(T_\mu^{\downarrow}).
	\end{eqnarray}
\end{lemma}

\begin{proof}
	If we can show the statements hold for any partition $\mu$ which is formed from $\lambda$ by 
	moving only one box then inductively it will hold for all $\lambda \trianglerighteq \mu$.
	Suppose $\mu$ is formed from $\lambda$ by moving a box from row $a$ to row $b$, with $a<b\leq l(\lambda)+1$
	(if $b=l(\lambda)+1$ then a new row is created by placing the removed box on the very bottom of the diagram). 
	The box we move goes from coordinates $(a,\lambda_{a})$ of $\lambda$ to $(b,\lambda_{b}+1)$ in $\mu$. 
	
	We shall prove first that $	\eig(\best) \geq  
	\eig(T_\mu^{\rightarrow})$. 
	Since $\best$ and  $T_\mu^\rightarrow$ are both numbered from left to right, the effect of moving a box from row $a$ to row $b$ is that $T_\mu^{\rightarrow}(i,j) = \best(i,j) -1$ for any box $(i,j)$ with $a<i\le b$; boxes in all other rows contain the same values in both tableaux. Using equation~\eqref{eigenvaluesum}, and remembering to include a term to account for the box being moved, we find that:
	\begin{align*}
		n(\eig(\best)  -
		\eig(T_\mu^{\rightarrow})) 
		& =  
		\left(\frac{\lambda_{a}-a+1}{\best(a,\lambda_{a})} 
		-\frac{(\lambda_{b}+1)-b+1}{T_\mu^{\rightarrow}(b,\lambda_{b}+1)}\right)  
		+ \sum_{\substack{(i,j) \in T_\lambda\cap T_\mu \\ \textnormal{with } a < i 
		\leq b}}\left[\frac{1}{\best(i,j)} - 
		\frac{1}{T_\mu^{\rightarrow}(i,j)}\right](j-i+1)\\
%		& \geq & \left(\frac{1-a+\lambda_{a}}{\Lambda_{a}} 
%		-\frac{2-b+\lambda_{b}}{\Lambda_{b}}\right)  
%		+ \sum_{\substack{\textnormal{boxes } (i,j) \\ \textnormal{with } a < i 
%				\leq b}} \left( \frac{1}{\best(i,j)} - 
%		\frac{1}{T_\mu^{\rightarrow}(i,j)}\right)(1-a+\lambda_{a})\\
		& \geq  \left( \frac{\lambda_{a}-a+1}{\best(a,\lambda_{a})} 
		-\frac{(\lambda_{b}+1)-b+1}{T_\mu^\rightarrow(b,\lambda_b+1)}\right)  + (\lambda_a-a+1)
		\left(\frac{1}{T_\mu^{\rightarrow}(b,\lambda_{b}+1)} - 
		\frac{1}{\best(a,\lambda_{a})}\right)\\
		& =  \frac{(\lambda_{a} -\lambda_{b}) + 
			(b-a) -1}{T_\mu^{\rightarrow}(b,\lambda_{b}+1)} 	\geq 0.
	\end{align*} 
	The first inequality holds because all the square-bracketed terms in the sum are negative; 
	we upper bound $j-i+1 \le \lambda_{a}-a+1$, and the resulting sum telescopes. The final inequality holds because $(\lambda_{a}-\lambda_{b})\ge 1$ and $(b-a)\ge 1$.

	For the second inequality, recall that $\lambda \trianglerighteq \mu$ if and only if $\mu^{\prime} \trianglerighteq \lambda^{\prime}$. 
	Therefore, using the first established inequality we find that $\eig(T_{\mu^{\prime}}^{\rightarrow}) \geq \eig(T_{\lambda^{\prime}}^{\rightarrow})$. 
	Now Lemma~\ref{lem:sumtranspose} gives $-\eig(T_{\mu}^{\downarrow}) \geq -\eig(T_{\lambda}^{\downarrow})$ and thus we recover the desired inequality.
\end{proof}

\subsection{Upper Bound Analysis}
\label{UPPERBOUNDANALYSIS}
In this section we complete the proof of the upper bound of Theorem \ref{LR main theorem}, making use of the results of Section~\ref{EIGENVALUES}. 
The analysis splits into two parts, dealing separately with those partitions $\lambda$ having either ``large'' or ``small'' first row.

Lemma~\ref{ClassicL2} allows us to upper bound the total variation distance in terms of the non-trivial eigenvalues of the transition matrix. Using Lemma~\ref{compute} we see that the trivial eigenvalue corresponds to the one-dimensional partition $\lambda=(n)$, and so Theorem~\ref{Maintheorem} implies that 
\[ 4\lVert \LR^{t}_{n} -\pi_{n} \rVert_{\textnormal{TV}}^{2}\,\le\, \sum_{\substack{\lambda\vdash n\\\lambda \neq (n)}} 
d_{\lambda} \sum_{T \in 
	SYT(\lambda)} 
\eig(T)^{2t} \,.
 \]
Recall from Lemma~\ref{seq} that for any $T \in \SYT(\lambda)$ the eigenvalue corresponding to $T$ may be bounded by those corresponding to $\worst$ and $\best$. With this in mind, we let $\Lambda_n^\rightarrow = \{\lambda\vdash n\,:\, |\eig(T_\lambda^\downarrow)|\,\le\, \eig(T_\lambda^\rightarrow)\}$ and $\Lambda_n^\downarrow = \{\lambda\vdash n\,:\, |\eig(T_\lambda^\downarrow)|\,>\, |\eig(T_\lambda^\rightarrow)|\}$; note that these are disjoint sets, with $\Lambda_n^\rightarrow \subseteq \{\lambda\vdash n\,:\, \eig(\best) \,\ge\,0 \}$ and $\Lambda_n^\downarrow \subseteq \{\lambda\vdash n\,:\, \eig(\worst) \,<\,0 \}$. Using Lemma~\ref{seq} and then Corollary~\ref{flip} we relax the upper bound as follows:
\begin{eqnarray}
	4\lVert \LR^{t}_{n} -\pi_{n} \rVert_{\textnormal{TV}}^{2} & \leq & \eig(T_{(1^n)})^{2t} \,+ \,
	\sum_{\substack{\lambda\in \Lambda_n^\rightarrow\\\lambda \neq(n)}} 
	d_{\lambda} \sum_{T \in SYT(\lambda)} \eig(T)^{2t}
	+ \sum_{\substack{\lambda\in \Lambda_n^\downarrow\\\lambda \neq(1^n)}} 
	d_{\lambda} \sum_{T \in SYT(\lambda)} \eig(T)^{2t}\nonumber \\
	& \leq &   \eig(T_{(1^n)})^{2t} \,+ \,
		\sum_{\substack{\lambda\in \Lambda_n^\rightarrow\\\lambda \neq(n)}} 
	d_{\lambda}^2 \,\eig(\best)^{2t}
	+ \sum_{\substack{\lambda\in \Lambda_n^\downarrow\\\lambda \neq(1^n)}} 
	d_{\lambda}^2 \,\eig(\worst)^{2t} \nonumber\\	
	& \leq & \eig(T_{(1^n)})^{2t} \,+ \,
	 \sum_{\substack{\lambda\,:\,\eig(\SEQ) \geq 0\\ \lambda \neq (n)}} 
	d_{\lambda}^{2}  \,	\eig(\best)^{2t} 
	+ \sum_{\substack{\lambda\,:\,\eig(\OSEQ) < 0 \\ \lambda \neq(1^n)}} 
	d_{\lambda}^{2} \,\eig(\worst)^{2t} \nonumber\\
			& \leq &   \eig(T_{(1^n)})^{2t} \,+ \,
	 \sum_{\substack{\lambda\,:\,\eig(\SEQ) \geq 0\\ \lambda \neq (n)}} 
	d_{\lambda}^2 \,\eig(\best)^{2t}
	+ \sum_{\substack{\lambda\,:\,\eig(\OSEQ) < 0 \\ \lambda'\neq(n)}} 
	d_{\lambda'}^2 \,\eig(T^\rightarrow_{\lambda'})^{2t} \nonumber \\
	& \leq & \eig(T_{(1^n)})^{2t} \,+ \,
	 2\,\sum_{\substack{\lambda\,:\,\eig(\SEQ) \geq 0\\ \lambda \neq (n)}} 
	 	d_{\lambda}^{2}  \,	\eig(\best)^{2t} \,. \label{eqn:UB_eigs}
\end{eqnarray}
(In the penultimate line we have used Corollary~\ref{flip} and the fact that $d_{\lambda'}=d_\lambda$. The final inequality follows by a second application of Corollary~\ref{flip}: if $\lambda$ satisfies $\eig(\OSEQ)<0$ then $\eig(T^\rightarrow_{\lambda'})$ must be non-negative.)

The first term in \eqref{eqn:UB_eigs} is simple to deal with at time $t=n\log(n)+cn$. We have already observed that $\eig\left(T_{(n)}\right) = 1$, and so Lemma~\ref{lem:sumtranspose} implies that $\eig\left(T_{(1^{n})}\right) = 2 H_{n}/n-1$. This means that
\begin{equation}\label{eqn:first_partition}
\eig\left(T_{(1^n)}\right)^{2t} = 
\left(1 - 
\frac{2H_{n}}{n} \right)^{2(n\log n + cn)} 
\end{equation}
and, using the bound $1-x\le e^{-x}$, we see that this tends to zero for any fixed $c$ as $n\to\infty$.

%Note that for $\lambda=(1^{n})$ we have $\SEQ = \OSEQ$. 
It therefore remains to bound the sum in \eqref{eqn:UB_eigs}. The partitions with 
the biggest eigenvalues will be those with large first rows $\lambda_{1}$, and so we split the analysis into two parts according to this value; by \emph{large} partitions we mean those with 
$\lambda_{1} \geq 3n/4$, and \emph{small} partitions are those with $\lambda_{1} <3n/4$. Large partitions give the biggest eigenvalues for $\LR_{n}$ and must be dealt with carefully; it is these which will determine the mixing time of the shuffle. Small partitions have correspondingly large dimensions, but eigenvalues which are small enough to give control around any time of order $n\log(n)$. 
We begin by identifying the partition at the top of the dominance ordering for any fixed value of $\lambda_1$, which allows us to employ Lemma~\ref{order}. If $\lambda\vdash n$ has first row equal to $\lambda_{1}=n-k$, then by moving boxes up and to the right it follows trivially that  
	\[ \lambda \trianglelefteq \begin{cases}
	(n-k,k) & \text{ if } k \leq \frac{n}{2}\\
	(n-k,n-k,\dots) = (n-k,\star) & \text{ if } \frac{n}{2} <k\leq n-1  \,,
	\end{cases}
	\]
where we write $(n-k,\star)$ for the partition which has as many rows of $n-k$ 
boxes as possible, with the last row being formed from whatever is left over; 
it will transpire that in this case only the size of the first two rows will be 
important for our bounds.

For each $k$ we also need a bound on sum of the squared dimensions of all partitions with 
$\lambda_{1}=n-k$, and for this we use: 

\begin{lemma}[Corollary 2 of \cite{diaconis1981generating}]\label{lem:Diaconis}
	\[\sum_{\substack{\lambda \vdash n \\ \lambda_{1} =n-k} }
	d_{\lambda}^{2} \leq {n \choose k}^{2} k!\leq   \frac{n^{2k}}{k!} \,.\]
\end{lemma}

\subsubsection*{Large Partitions}

Let $\lambda$ be a partition satisfying $\eig(\best)\ge 0$, and for which $\lambda_1 = n-k$ for some  $k\leq n/4$. We have observed above that $\lambda \trianglelefteq(n-k,k)$, and so Lemma~\ref{order} suggests that we look at the eigenvalue of $T_{(n-k,k)}^{\rightarrow}$. Using our eigenvalue formula from Lemma~\ref{compute} we calculate this as follows, with the first/second sum corresponding to the first/second row of $T_{(n-k,k)}^{\rightarrow}$:
\begin{align}
\eig(T_{(n-k,k)}^{\rightarrow}) & = \frac 1n \sum_{j=1}^{n-k} \frac{j}{T_{(n-k,k)}^{\rightarrow}(1,j)} 
	\,+\, \frac 1n\sum_{j=1}^{k} 
	\frac{j-1}{T_{(n-k,k)}^{\rightarrow}(2,j)} 
	\,=\, \frac{n-k}{n} \,+\, \frac 1n \sum_{j=1}^{k} 
	\frac{j-1}{n-k+j} \label{eqn:large} \\
	&=  1 -\frac{(n-k+1)}{n}(H_{n}-H_{n-k+1}) - \frac 1n \,. \nonumber
\end{align}

We now use this, along with the inequality $1-x \leq e^{-x}$, to bound the contribution of large partitions to the sum in \eqref{eqn:UB_eigs}:

\begin{align}
	\sum_{k=1}^{n/4}
	\sum_{\substack{\lambda\,:\,\eig(\best)\ge 0\\ \lambda_{1}=n-k}} 
	d_{\lambda}^{2} \text{eig}(T_{\lambda}^{\rightarrow})^{2t}
	& \, \leq\,   	\sum_{k=1}^{n/4} \frac{n^{2k}}{k!}\,
	\text{eig}\left(T_{(n-k,k)}^{\rightarrow}\right)^{2t}  \qquad\text{(by Lemma~\ref{lem:Diaconis})} \nonumber \\
		& \,\leq\, 	\sum_{k=1}^{n/4} \frac{n^{2k}}{k!}\, \left( 1 
	-\frac{(n-k+1)}{n}(H_{n}-H_{n-k+1}) - \frac{1}{n}\right)^{2t} \nonumber \\
	& \,\leq\, 	\sum_{k=1}^{n/4}  \frac{n^{2k}}{k!}\, e^{-2t\left(\frac{(n-k+1)}{n}(H_{n}-H_{n-k+1})+ 
		\frac{1}{n}\right)} \nonumber \\
	& \,\leq\,  e^{-2c}	\sum_{k=1}^{n/4}
	\frac{n^{2k-2(n-k+1)(H_{n}-H_{n-k+1})-2}}{k!}\,, \label{sum_large}
\end{align}
where in the last step we have substituted $t=n\log n + cn$.

The ratio of $(k+1)^{th}$ to $k^{th}$ terms in this sum is given by 
\begin{eqnarray}
\label{ratio1}
\frac{n^{2(H_n-H_{n-k})}}{k+1}\,.
%\frac{n^{2 - 2(n-l)\log\left(\frac{n-l+1}{n-l}\right) 
%+2\log\left(\frac{n}{n-l+1}\right)}}{l+1} \approx
%\frac{n^{2\log\left(\frac{n}{n-l+1}\right)}}{l+1}.
\end{eqnarray}
For large $n$ this ratio is less than one for all $k=1,\dots,n/4$. Indeed, a little analysis shows that for large $n$ the largest value of the ratio over this range of $k$ is achieved when $k=1$, at which point it equals $n^{2/n}/2$. For sufficiently large $n$ this ratio is thus bounded above by $3/4$, say, which permits us to bound the sum in \eqref{sum_large} by a geometric series with initial term 1:
\begin{equation}
\label{eqn:sum_large_1}
e^{-2c}\,	\sum_{k=1}^{n/4}
\frac{n^{2k-2(n-k+1)(H_{n}-H_{n-k+1})-2}}{k!} \, \le \, e^{-2c}	
\sum_{k=1}^{n/4} (3/4)^{k-1} \,\le\, 4e^{-2c}\,.
\end{equation}

\subsubsection*{Small partitions}

Now consider a partition $\lambda$ satisfying $\eig(\best)\ge 0$ and for which 
$\lambda_1 = n-k$ with  $n/4<k\le n-2$. Suppose first of all that $n/4<k\le n/2$; as in the large partition case, any such partition is dominated by $(n-k,k)$, and the same calculation as in equation \eqref{eqn:large} shows that 
\begin{equation}\label{eqn:medium}
\eig\left(T_{(n-k,k)}^{\rightarrow}\right)  
\,=\, \frac{n-k}{n} \,+\, \frac 1n \sum_{j=1}^{k} 
\frac{j-1}{n-k+j} \,. 
\end{equation}

Now consider the case when $k>n/2$. Here we have already identified that 
$\lambda\trianglelefteq(n-k,\star)$, and so we proceed by calculating the 
eigenvalue of $T^\rightarrow_{(n-k,\star)}$. Note first that for any box 
$(i,j)$ with $i\ge 3$, 
\begin{align*}
\frac{j-i+1}{T_{(n-k,\star)}^{\rightarrow}(i,j)} &= \frac{j-i+1}{(i-1)(n-k) + 
j} \le \frac{(n-k)}{(i-1)(n-k) + (n-k)} \le \frac 13\,.
\end{align*}
Using this inequality in conjunction with Lemma~\ref{compute} we bound 
$\eig(T_{(n-k,\star)}^{\rightarrow})$ as follows:
\begin{align}
\eig(T_{(n-k,\star)}^{\rightarrow}) & = \frac 1n \sum_{j=1}^{n-k} 
\frac{j}{T_{(n-k,\star)}^{\rightarrow}(1,j)} 
\,+\, \frac 1n\sum_{j=1}^{n-k} 
\frac{j-1}{T_{(n-k,\star)}^{\rightarrow}(2,j)} \, + \, \frac 1n 
\sum_{\substack{(i,j) \nonumber \\ 
		i\geq 3}} \frac{j-i+1}{T_{(n-k,\star)}^{\rightarrow}(i,j)} \\
& \leq \frac{n-k}{n} \,+\, \frac 1n \sum_{j=1}^{n-k} 
\frac{j-1}{n-k+j} \, + \,  \frac{n-2(n-k)}{3n}\,. \label{eqn:small1}
\end{align}
We now observe that \eqref{eqn:small1} provides an upper bound for the expression in \eqref{eqn:medium}. Indeed, for $n/4<k\le n/2$ we may write
\begin{align*}
\frac{n-k}{n} \,+\, \frac 1n \sum_{j=1}^{n-k} 
\frac{j-1}{n-k+j} \, + \,  \frac{n-2(n-k)}{3n} - \eig(T_{(n-k,k)}^{\rightarrow}) &\,=\, \frac 1n \sum_{j=k+1}^{n-k} \left(\frac{j-1}{n-k+j} -\frac 13 \right) \\
&\,= \,  \frac{2(n-2k)}{3n} - \frac{(n-k+1)}{n}(H_{2(n-k)}-H_n) \,.
\end{align*}

Substituting $k=\gamma n$, this final expression is bounded below for any $n\ge 15$ by the function $f(\gamma)$, where $f:[1/4,1/2]\to \mathbb{R}$ is defined by
\[ f(\gamma) = \frac{2(1-2\gamma)}{3} -\left(1-\gamma + \frac{1}{15}\right)\log(2(1-\gamma)) \,. \]
This function is non-negative for all $\gamma\in[1/4,1/2]$, thus completing our claim.

We have just shown that for any $\lambda \vdash n$ satisfying $\eig(\best)\ge 0$ and for which $\lambda_{1}=n-k$ with $n/4 < k\leq n-2$,
\begin{align*}
\eig\left(T_{\lambda}^{\rightarrow}\right) \,&\le\, \frac{n-k}{n} \,+\, \frac 1n \sum_{j=1}^{n-k} 
\frac{j-1}{n-k+j} \, + \,  \frac{n-2(n-k)}{3n} \\
&= \frac{n-k}{n} + \frac{n-k-1 - (n-k+1)(H_{2(n-k)} -H_{n-k +1})}{n} +  
\frac{2k-n}{3n} \\ 
& = 1 -\frac{(4k-2n+3)}{3n} - \frac{(n-k+1)}{n}(H_{2(n-k)} -H_{n-k +1})\,.
\end{align*}
Using the inequalities $1-x \leq e^{-x}$ for all $x$, and $(x+1)\left(H_{2x} -H_{x+1}\right) > (x-1)\log 2$ for all integers $x\geq 2$, 
%\texttt{(OMR: proof of this commented out at the end)} 
we are able to bound the 
contribution from small partitions to the sum in \eqref{eqn:UB_eigs} at time $t=n\log n +cn$ as follows:
\begin{align}
\sum_{k=n/4}^{n-2}
\sum_{\substack{\lambda\,:\,\eig(\best)\ge 0\\ \lambda_{1}=n-k}} 
d_{\lambda}^{2} \eig\left(T_{\lambda}^{\rightarrow}\right)^{2t}
& \, \leq\,   	\sum_{k=n/4}^{n-2} \frac{n^{2k}}{k!}\,
e^{-\frac{2t}{n}\left(\frac{4k-2n+3}{3} + (n-k+1)(H_{2(n-k)} -H_{n-k +1})\right)}\nonumber \\
%	\left(1 
%	+\left(\frac{2}{3}-h_{n-l}\right) 
%	-\left(\frac{4}{3}-h_{n-l}\right)\frac{l}{n} 
%	-\frac{(1+h_{n-l})}{n}   \right) ^{2t} \cdot 
%	\frac{n^{2l}}{l!}\\
%	& \leq \sum_{l >n/2}^{n-2} 
%	n^{2(\frac{2}{3}-h_{n-l})n -2(\frac{4}{3}-h_{n-l})l}n^{-2(1 +h_{n-l})} 
%	\cdot
%	\frac{n^{2l}}{l!} \cdot 
%	e^{2c(\left(\frac{2}{3}-h_{n-l}\right)n
%		-\left(\frac{4}{3}-h_{n-l}\right)l 
%		-(5/4))}\\
%	& \leq   
%	e^{-2c}\sum_{k =n/4}^{n-2} 
%	\frac{n^{\frac{4n-2k-6}{3}-2(n-k+1)(H_{2(n-k)} -H_{n-k +1})}}{k!}\,. 
%			\label{sum_small}
	& \leq   
	e^{-2c}\sum_{k =n/4}^{n-2} 
	\frac{n^{\frac{4n-2k-6}{3}-2(n-k-1)\log 2}}{k!}\,. 
	\label{sum_small}
\end{align}
Once again writing $k=\gamma n$, now for $\gamma\in[1/4,1]$, a straightforward application of Stirling's formula shows that for large $n$ the dominant term in the summand of \eqref{sum_small} takes the form $n^{g(\gamma)n/3}$,
%\[ (2\pi\gamma)^{-1/2}(e/\gamma)^{\gamma n} n^{g(\gamma)\frac n3 - \frac 52} \,, \]
where $g(\gamma) = 4-5\gamma - 6(1-\gamma)\log 2 < 0$ for all $\gamma\in[1/4,1]$. It follows that, for any fixed $c$, the sum in \eqref{sum_small} tends to zero as $n\to\infty$. Combining this result with the bounds in \eqref{eqn:UB_eigs}, \eqref{eqn:first_partition} and \eqref{eqn:sum_large_1} completes the proof of the upper bound on the mixing time in Theorem~\ref{LR main theorem}.

\section{Lower Bound}
\label{LOWERBOUND}

To complete Theorem \ref{LR main theorem} we need to prove the lower bound on 
the mixing time. To do this we employ the usual trick of finding a set of permutations $B_n\subset S_n$  which has significantly different probability under the equilibrium distribution $\pi_n$ and the one-sided transposition measure $\RL_n^t$. The definition of total variation distance then immediately yields a simple lower bound:
\[ \lVert \RL_n^t -\pi_n \rVert_{\textnormal{TV}}\, \ge\, \RL_n^t(B_n) - \pi_n(B_n) \,. \]
In particular, we follow in the steps of~\cite{diaconis1981generating} and find a suitable set $B_n$ by considering the number of fixed points within (a certain part of) the deck. Estimation of $\RL_n^t(B_n)$ then reduces to a novel variant of the classical coupon collector's problem.

Recall that the deck of $n$ cards are labelled $\{1,\ldots,n\}$ from bottom to top.  One step of the one-sided transposition shuffle on $n$ cards may be modelled by 
firstly choosing a position $\righthand \sim U\{1,\dots n\}$ with our right 
hand, and then choosing a position (below our right hand) $\lefthand\sim U\{1,\dots ,\righthand\}$ with our left hand and transposing the cards in the chosen positions.

Since the left hand always chooses a position below that of the right hand, it is intuitively clear that our shuffle is relatively unlikely to transpose two cards near to the top of the deck. This leads us to focus attention on a set of positions at the top of the deck:  write $V_n$ for the top part of the deck, where 
\[ V_n = \{n-n/m+1,\dots,n-1,n\} \,, \]
and where $m=m(n)$ is to be chosen later. We shall be keeping track of fixed points within this part of the deck. Let 
\[ B_n = \{\rho\in S_n \, | \, \text{$\rho$ has at least 1 fixed point in $V_n$}\} \,.\]
Note that $V_n$ contains $n/m$ positions, and so we may upper bound the size of $B_{n}$ by choosing one position in $V_{n}$ to fix and considering all permutations of the other $n-1$ positions. This shows that $|B_{n}| \leq (n/m)(n-1)!$, and hence $\pi_{n}(B_n) \leq 1/m$.

To bound the value of $\RL_n^t(B_n)$ we  
reduce the problem to studying a simpler Markov chain linked to coupon 
collecting. When either of our hands $(\righthand, \lefthand)$ picks a new 
(previously 
untouched) card we shall say that this card gets \emph{collected}. The  
uncollected cards in $V_n$ at time $t$ are those which have not yet been picked 
by either hand, and thus the size of this set gives us a lower bound on the 
number of fixed points in $V_n$. Writing $U_n^t$ for the set of uncollected 
cards in $V_n$ after $t$ one-sided transposition shuffles, it follows that 
\begin{eqnarray}
\label{Ineq1}
\RL_{n}^{t}(B_n) \,\geq\, \mathbb{P}(|U^{t}_{n}| \ge 1 )\,.
%\mathbb{P}(\{\righthand,\lefthand\}_{i=1}^{t} \text{have 
%	at least 1 uncollected card}).
\end{eqnarray}

We wish to show that at time $t=n\log n - n\log\log n$ the probability on the 
right hand side of \eqref{Ineq1} is large. The difficulty with the analysis 
here is that in the one-sided transposition shuffle the value of $L^i$ is 
clearly not independent of $R^i$. This means that a standard 
coupon-collecting argument for the time taken to collect all of the 
cards/positions in $V_n$ cannot be applied in our setting, and a little more 
work is therefore required.

Note that at each step there are four possibilities: both hands collect new cards, only one hand does (left or right) or neither does. This permits us to bound the change in the number of \emph{collected} cards as follows:
\begin{align}
|V_n\setminus U_n^{t+1}| & \,=\, |V_n\setminus U_n^{t}| + 
|\{L^{t+1},R^{t+1}\}\cap U_n^t| \nonumber \\
& \,\le\, |V_n\setminus U_n^{t}| + 2 \cdot \ind [L^{t+1}\in U_n^t] + 
\ind [L^{t+1}\notin U_n^t,R^{t+1}\in U_n^t]\,,\label{eqn:indicator_fns} 
\end{align}
where $\ind[\cdot]$ is an indicator function. Now, since the left hand is more likely to choose positions towards the bottom of the pack, 
\begin{align*}
\mathbb{P}(L^{t+1}\in U_n^t) &\,\le\, \mathbb{P}(L^{t+1}\in \hat U_n^t) \,,
\end{align*}
where $\hat U_n^t = \{n-n/m+1,\dots,n-n/m+ |U_n^t|\}$, i.e. the $|U_n^t|$ lowest numbered positions in $V_n$. Furthermore, 
\begin{align}
\mathbb{P}(L^{t+1}\in \hat U_n^t) &\,=\, \frac 1n\sum_{k\in \hat 
U_n^t}\mathbb{P}\left(L^{t+1}\in \hat U_n^t\,|\, R^{t+1}=k\right)+  \frac 
1n\sum_{k\in V_n\setminus \hat U_n^t}\mathbb{P}\left(L^{t+1}\in \hat U_n^t\,|\, 
R^{t+1}=k\right) \nonumber \\
&\,=\, \frac 1n\sum_{k\in \hat U_n^t} \frac{k-(n-n/m)}{k} +  \frac 1n\sum_{k\in 
V_n\setminus \hat U_n^t}\frac{|U_n^t|}{k} \nonumber \\
&\,\le\, \frac{\sum_{k=1}^{|U_n^t|} k + (n/m -|U_n^t|)|U_n^t|}{n(n-n/m)} 
\,\le\,\frac{|U_n^t|}{(m-1)n}\,. \label{eqn:bound_minus2}
\end{align}

The probability of the final event in \eqref{eqn:indicator_fns} is simple to bound:
\begin{equation}\label{eqn:bound_minus1}
\mathbb{P}(L^{t+1}\notin U_n^t,R^{t+1}\in U_n^t) \,\le\, \mathbb{P}(R^{t+1}\in U_n^t) \,\le\, \frac{|U_n^t|}{n}\,.
\end{equation}

Using \eqref{eqn:indicator_fns}, \eqref{eqn:bound_minus2} and \eqref{eqn:bound_minus1} together, we now define a counting process $M_n^t$ which stochastically dominates the number of collected cards $|V_n\setminus U_n^{t}|$ at all times: 
\begin{align}
M_n^0 \,&=\, 0 \,; \nonumber \\
\mathbb{P}({M}^{t+1}_{n} = {M}^{t}_{n}+k) \,&=\, 
\begin{cases} 
\frac{1}{(m-1)n}\left(\frac{n}{m}-{M}^{t}_{n}\right) & \text{ 
	if } k=2\\
\frac{1}{n}\left(\frac{n}{m}-{M}^{t}_{n}\right) & \text{ if } k =1\\
1- \frac{m}{(m-1)n}\left(\frac{n}{m}-{M}^{t}_{n}\right) & 
\text{ if } k =0\,. \label{eqn:M_change}
\end{cases}
\end{align}

Combining this with \eqref{Ineq1} we obtain the following bound on $\LR^{t}_{n}(B_n)$:
\begin{equation}
\label{Ineq3}
\RL^{t}_{n}(B_n)  \geq 
\mathbb{P}\left(M^{t}_{n} 
< n/m\right) \,.
% \geq  \mathbb{P}\left(\hat{M}^{t}_{n} < (n/m)\right). 
\end{equation}

We are interested in the time at which the process ${M}^{t}_{n}$ first reaches level $n/m$, where we now take $m=m(n) =\log n$. 
\begin{lemma}\label{lem:lower_bound}
	Let $\mathcal{T} = \min\{t\,:\, M_n^t \ge n/\log n\}$. Then for any $c>2$,
	\[ \lim_{n\rightarrow \infty}\mathbb{P}(\mathcal{T}<n\log n - n \log \log n 
	-cn)\, \leq \, \frac{\pi^{2}}{6(c-2)^2} \,. \] 
\end{lemma}
%\texttt{SBC: I've got a slightly different bound here, with $(c-2)^2$ in the 
%denominator. Needs checking (and then propagating, if necessary).}
Before proving Lemma~\ref{lem:lower_bound} we first show how this result quickly leads to a proof of the lower bound in Theorem~\ref{LR main theorem}. Writing $t = n\log n - n \log \log n $ we obtain
\begin{align*}
\lVert \RL_n^{t-cn} -\pi_n \rVert_{\textnormal{TV}}\, &\ge\, \RL_n^{t-cn}(B_n) - \pi_n(B_n) 
\,\ge\, \mathbb{P}\left(M^{t-cn}_{n} < n/m(n)\right)- 1/m(n)\\
&= \, \mathbb{P}(\mathcal{T}>t-cn) - \frac{1}{\log n } 
\,\ge\,  1 -\frac{\pi^{2}}{6(c-2)^2} - \frac{1}{\log n } \,,
\end{align*}
as required.

\begin{proof}[Proof of Lemma~\ref{lem:lower_bound}]
	Let $\mathcal{T}_{i}$ be the time spent by the process 
	${M}^{t}_{n}$ in each state $i\ge0$. We want to find $\mathcal{T} = \mathcal{T}_{0} + \mathcal{T}_{1} + \dots + 
	\mathcal{T}_{(n/m)-1}$.
	
	From \eqref{eqn:M_change} we see that 
	\begin{equation}\label{eqn:M_up}
	p_i\,:=\,\mathbb{P}({M}^{t+1}_{n} > {M}^{t}_{n}\,|\, M_n^t = i) \,=\, 
	\frac{m}{(m-1)n}\left(\frac{n}{m}-i\right)\,.
	\end{equation}
	In the standard coupon collecting problem each of the random variables $\mathcal T_i$ has a geometric distribution with success probability $p_i$. Here, however, we have to take into account the chance that our counting process $M_n$ increments by two, leading it to spend zero time at some state. Note first that 
	\[\mathbb{P}({M}^{t+1}_{n} ={M}^{t}_{n} + 2 \,|\, {M}^{t+1}_{n} > {M}^{t}_{n}) \,=\, \frac{1}{m}\,,\]
	independently of the value of $M_n^t$. 
	Prior to spending any time in state $i$, the process ${M}^{t}_{n}$ must visit (at least) one of the states $i-1$ or $i-2$. A simple argument shows that 
	\[ \mathbb{P}(\mathcal{T}_{i}>0\, | \,\mathcal{T}_{i-1}>0) \,=\, 1- \frac{1}{m}\,,\quad\text{and}\quad
	\mathbb{P}(\mathcal{T}_{i}>0\, |\, \mathcal{T}_{i-2}>0) \,=\, 1- 
	\frac{1}{m}\left(1-\frac{1}{m}\right) \ge 1-\frac 1m\,. \]
	Therefore $\mathbb{P}(\mathcal T_i>0) \ge 1-\frac{1}{m}$ for all states $i$, and so $\mathcal{T}_i$ stochastically dominates the random variable $\mathcal T'_i$ with mass function
	\begin{equation}\label{eqn:Tprime}
	\mathbb P(\mathcal T'_i = k) = \begin{cases}
	1/m &\quad k=0 \\
	(1-1/m)p_i(1-p_i)^{k-1} &\quad k\ge 1\,.
	\end{cases}
	\end{equation}
	It follows that $\mathbb P(\mathcal T<t) \le \mathbb P(\mathcal T'<t)$ for any $t$, where $\mathcal T' = \mathcal{T}'_{0} + \mathcal{T}'_{1} + \dots + \mathcal{T}'_{(n/m)-1}$.
	Setting $m = m(n)=\log n$ we may bound the expectation and variance of $\mathcal T'$:
	\begin{eqnarray*}
		\E[\mathcal{T}^{\prime}] & = & 
		\sum_{i=0}^{n/m-1} 
		\frac{m-1}{mp_i} =  \left(\frac{m-1}{m}\right)^{2} n\log(n/m)\,\geq\, n\log n  
		-n\log \log n -2n \,;\\
		\Var[\mathcal{T}^{\prime}] & \le & \sum_{i=0}^{n/m-1}  \frac{1}{p_i^2} \,\le\, 
		\sum_{i=1}^{n/m}\frac{n^{2}}{i^{2}} 
		\,\leq\, 	\frac{\pi^{2}}{6}n^{2}\,.
	\end{eqnarray*}
	Finally, applying Chebyshev's inequality yields the following for any $c>2$:
	\[	\mathbb{P}\left(\mathcal{T}^{\prime} < n\log(n) - n \log \log n -cn\right) 
	\,\leq\, 
	\mathbb{P}\left(|\mathcal{T}^{\prime} -\E[\mathcal{T}^{\prime}]\,|\, >  (c-2)n \right)
	\,\leq\,  \frac{\pi^2}{6(c-2)^2}\,.
	\]
\end{proof}

\section{Biased One-sided Transpositions}
\label{GENERALISATIONS}

In this section we generalise the result of Theorem~\ref{LR main theorem} by allowing the right hand to choose from a more general distribution on $[n]$.

\begin{defn}\label{def:weight}
	Given a \emph{weight function} $w:\mathbb{N} \rightarrow (0,\infty)$, let $\NW = \sum_{i=1}^{n} w(i)$ denote the cumulative weight up to $n$. Then the \emph{biased one-sided transposition shuffle} $\BLR_{n,w}$ is the random walk on $S_n$ generated by the following distribution on transpositions:
	\begin{eqnarray}
	\BLR_{n,w}(\tau) = 
	\begin{cases}
	\frac{w(j)}{\NW}\cdot\frac{1}{j} &  \text{if } \tau = (i\,j) 
	\text{ for some } 
	1\leq i 
	\leq j 
	\leq n\\
	0 & \text{otherwise.}
	\end{cases}
	\end{eqnarray}
\end{defn}

This shuffle allows a general distribution for the position chosen by the right hand, $R^i$, with
\begin{eqnarray}
\mathbb{P}(\righthand=j) = \frac{w(j)}{\NW}, \text{ for } 1 \leq j 
\leq 
n\,,
\end{eqnarray}
but maintains the property that the left hand $L^i$ chooses a position uniformly on the set $\{1,\dots,R^i\}$. Importantly the weight of each position $j$ may only 
depend on $j$ and not the size of the deck $n$. This setup implies that the 
biased shuffle preserves the recursive algebraic structure identified in the Appendix, and that the 
results of Theorems \ref{master} and \ref{liftappendix} still hold (up to minor changes in constants).
It follows that the eigenvalues of the biased one-sided transposition shuffles are 
still represented by standard Young tableaux, and that the eigenvalue associated 
to a tableau $T$ may be computed in a similar way.

\begin{lemma}	\label{biasedCompute}
	The eigenvalues for the biased one-sided transposition shuffle $P_{n,w}$ on $n$ cards 
	are indexed by standard Young tableau of size $n$. For a standard Young tableau $T$ of size $n$ and $m\in[n]$ define a function $T(m)$ by setting $T(m) = j-i+1$ if and only if $T(i,j) =m$. The 
	eigenvalue corresponding to a tableau $T$ is given by
	\begin{eqnarray*}
	\label{eigenvaluesum2}
	\eig(T) = \frac{1}{\NW}\sum_{\substack{\textnormal{boxes}\\ 
			(i,j)}} 
	\frac{j-i+1}{T(i,j)} \cdot w(T(i,j)) = \frac{1}{\NW}\sum_{m=1}^{n} 
	T(m) \frac{w(m)}{m}.
	\end{eqnarray*}
\end{lemma}

We now focus on a natural choice of weight function of the form $w(j) = 
j^{\alpha}$; we shall denote the resulting shuffle as $\BLR_{n,\alpha}$, and 
write $\NA$ in place of $\NW$. For $\alpha=0$ we recover the unbiased one-sided 
transposition shuffle $\LR_{n}$, while if $\alpha>0$ ($\alpha<0$) the right 
hand is biased towards the top (respectively, bottom) of the deck. 

\begin{thm}
	\label{thm:biasedbounds}
	Define the time $t_{n,\alpha}$ as, 
	\[t_{n,\alpha}= \begin{cases}
	\NA/n^{\alpha} & \textnormal{ if } \alpha \leq 1 \\
	\NA/N_{\alpha-1}(n) & \textnormal{ if } \alpha \geq 1 \,.
	\end{cases} \]
	The biased one-sided transposition shuffle $\BLR_{n,\alpha}$ exhibits 
	a total variation cutoff at time $t_{n,\alpha}\log n$ for all $\alpha\in \mathbb{R}$. 
	Specifically for any $c_{1} > 5/2, c_{2} >\max\{2, 3-\alpha\}$ we have:
	
	\begin{eqnarray*}
		\limsup_{n\rightarrow \infty}\, \lVert 
		\BLR_{n,\alpha}^{t_{n,\alpha}\left(\log n + c_{1}\right)} -\pi_{n} 
		\rVert & 
		\leq & 	Ae^{-2c_{1}} 
		\textnormal{ 	for a universal constant } A, \textnormal{ for all } 
		\alpha \\
		\text{ and } \quad 
		\liminf_{n\to\infty}\, \lVert 
		\BLR_{n,\alpha}^{t_{n,\alpha}\left(\log n - \log \log n - 
		c_{2}\right)} 
		-\pi_{n} \rVert & 
		\geq 
		& \begin{cases}
		1 - \frac{\pi^{2}}{6(c_{2}-3+\alpha)^{2}} & \textnormal{ for } \alpha 
		\leq 1 \\
		1 - \frac{\pi^{2}}{6(c_{2}-2)^{2}} & \textnormal{ for } \alpha \geq 1 \,.
		\end{cases} 
	\end{eqnarray*}
\end{thm}

The asymptotics of the cutoff times for the biased one-sided transposition shuffle are summarised in Table \ref{table:biasedtime}.

\begin{table}[H]
	\centering{
	\renewcommand{\arraystretch}{1.5}
	\begin{tabular}{c|c|c|c|c}
		& $\alpha \in (-\infty,-1)$ & $\alpha =-1$ & $\alpha \in (-1,1]$ & $\alpha \in (1,\infty)$
		\\\hline
		$t_{n,\alpha}\log n$ & $\zeta(-\alpha) n^{-\alpha}\log n$ & $n(\log n)^2$ & $\frac{1}{ 1+\alpha}n\log n$ & $\frac{\alpha}{1+\alpha}n\log n$ 
	\end{tabular}
	}
	\caption{Asymptotics of the cutoff time $t_{n,\alpha}\log n$ as $n\to\infty$, for different values of $\alpha$.}
	\label{table:biasedtime}
\end{table}

Note that the fastest mixing time is obtained when $\alpha 
=1$; using this weight function the shuffle is constant on the conjugacy class 
of transpositions, with transition probabilities similar to those of the classical random 
transpositions: $P_{n,1}((i \, j)) = 2/(n(n+1))$.
In this case we obtain a mixing time of $t_{n,1} \sim (n/2)\log n$ which agrees with that of 
the random transposition shuffle. The mixing time
is bounded above by $n\log n$ for all $\alpha>1$, but as $\alpha \to -\infty$ the 
mixing time is unbounded; in particular, when $\alpha<-1$ the mixing time 
is of order $O(n^{-\alpha}\log n)$. 

Theorem~\ref{thm:biasedbounds} is proved by generalising the results of 
Sections~\ref{UPPERBOUND} and \ref{LOWERBOUND}. Many of the arguments are 
almost identical to those already presented, so in what follows we shall simply 
sketch the main differences. We note that bounding the mixing time when using more general monotonic weight functions (but still satisfying Definition~\ref{def:weight}) is relatively straightforward: Lemma~\ref{biasedCompute} indicates that the upper bound on the mixing time is determined by whether $w(n)/n$ is increasing or decreasing in $n$. (See~\cite{MatheauRavenThesis} for further details.) Here we restrict attention to the case when $w(j)=j^\alpha$ since for this family of shuffles we are able to show the existence of a cutoff.

\subsection{Upper Bound for Biased One-sided Transpositions}
\label{GENUPPERBOUND}
We first of all note that the shuffles $\BLR_{n,\alpha}$ are still aperiodic, 
transitive, and reversible, meaning we may once again use Lemma~\ref{ClassicL2} 
to upper bound the mixing time. Furthermore, for $\alpha\leq 1$ the 
main results of Section~\ref{EIGENVALUES} still hold, meaning that it again 
makes sense to bound the eigenvalues of large and small partitions separately. 
For $\alpha \geq 1$ we will introduce a new tableau which will help us bound 
the eigenvalues for $P_{n, \alpha}$. After establishing bounds on our 
eigenvalues for all $\alpha$ we present a combined argument for the upper 
bound in Theorem \ref{thm:biasedbounds}.

\begin{lemma}
	\label{lem:genbound1}
	Let $\lambda \vdash n$ with $\lambda_{1} = n-k$. Then the eigenvalue 
	$\eig(\best)$ for the shuffle $\BLR_{n,\alpha}$ with $\alpha \leq1$ may be 
	bounded as follows:
	\[
	\eig(\best) \,\le\, 
	\begin{cases}
	1 - \frac{(n-k+1)kn^{\alpha}}{n\NA} & \quad \text{if $k \le n/4$} \\
	1 - \frac{k n^{\alpha}}{2\NA} & \quad \text{if $n/4<k$.} \\
	\end{cases}
	\]
\end{lemma}

\begin{proof}
	For $k\le n/2$, the maximum partition in the dominance order for the class 
	of 
	partitions with $\lambda_{1}=n-k$ is $(n-k,k)$, and so $\eig(\best) \le 
	\eig\left(T_{(n-k,k)}^{\rightarrow}\right)$. The eigenvalue of 
	$T_{(n-k,k)}^{\rightarrow}$ may be calculated by summing over the two rows 
	of the partition $(n-k,k)$ and 
	using Lemma \ref{biasedCompute}, as follows:
	\begin{eqnarray*}
			\NA \, 
		\textnormal{eig}\left(T_{(n-k,k)}^{\rightarrow}\right) & = &
		\sum_{\substack{m=1}}^{n-k} m^{\alpha} + \sum_{m=n-k+1}^{n} 
		(m-n+k-1)m^{\alpha-1} \\
		& = & \sum_{m=1}^{n} m^{\alpha} - (n-k+1) \sum_{m=n-k+1}^{n} 
		m^{\alpha-1} 
		\\
		& \leq & \NA -\frac{(n-k+1)kn^{\alpha}}{n}.
	\end{eqnarray*}
	This immediately proves the desired inequality for $k\le n/4$, and also 
	yields the stated bound for $n/4<k\le n/2$.
	
	For $k>n/2$ we once again need to bound 
	$\eig\left(T_{(n-k,\star)}^{\rightarrow}\right)$. Letting $\nu 
	=(n-k,\star)$ for ease of notation we calculate as follows:
	\begin{eqnarray*}
		\NA \, 
		\eig(T_{\nu}^{\rightarrow}) & = & \sum_{j=1}^{n-k} 
		j^{\alpha}  + \sum_{i=2}^{l(\nu)} 
		\sum_{j=1}^{\nu_{i}} (j-i+1)  ((i-1) (n-k)
			+j)^{\alpha-1}\\
		%		& = & \NA + \left(\sum_{i=2}^{l(\nu)} 
		%		\sum_{j=1}^{\nu_{i}} (j-i+1)  \frac{((i-1) (n-k) 
		%			+j)^{\alpha}}{(i-1) (n-k) +j} - \sum_{i=2}^{l(\nu)} 
		%		\sum_{j=1}^{\nu_{i}} ((i-1) (n-k)+j)^{\alpha}\right)\\
		& = & \NA - \sum_{i=2}^{l(\nu)} 
		\sum_{j=1}^{\nu_{i}} (i-1) (n-k+1) ((i-1) (n-k)+j)^{\alpha-1} \\
		%	& \leq & \NA - 
		%	\frac{n^\alpha}{n}\sum_{i=2}^{l(\nu)} 
		%	\sum_{j=1}^{\nu_{i}} \left((i-1)\cdot (n-l+1)\right)\\
		%		& = & \sum_{i=1}^{n} i^{\alpha} -\frac{n^{\alpha}}{n}\left( 
		%		\sum_{i =\nu_{1}+1}^{\nu_{2}} (n-l+1) + \sum_{i 
		%			=\nu_{2}+1}^{\nu_{3}} 2(n-l+1) + \dots + 
		%		\sum_{i =\nu_{l(\nu)-1}+1}^{n} (l(\nu)-1)(n-l+1) 
		%\right)\\
		& \leq & \NA -\frac{(n-k+1)n^{\alpha}}{n} 
		\sum_{i=2}^{l(\nu)} (i-1)  \nu_{i} \,.
	\end{eqnarray*}
	By definition of the partition $\nu$, each row but the last has size $n-k$, 
	and the final row has size $\nu_{l(\nu)} = n-(l(\nu)-1)(n-k)$. In addition, 
	since $l(\nu)=\lceil{n/(n-k)}\rceil$ we may write $l(\nu) = 
	n/(n-k)+x$ for some $0\le x<1$.	Substituting these values we obtain:
	\begin{eqnarray*}
		\NA \, 
		\eig(T_{\nu}^{\rightarrow}) 
		%		 \NA -\frac{(n-k+1)n^{\alpha}}{n} 
		%		\left((l(\nu)-1)\cdot 
		%		(n-(l(\nu)-1)\cdot(n-k))  + \sum_{i=2}^{l(\nu)-1}(n-k) (i-1) 
		%		\right)\\
		%		& = & \NA -\frac{(n-k+1)n^{\alpha}}{n} 
		%		\left( n(l(\nu)-1) - (n-k) 
		%		\frac{l(\nu)(l(\nu)-1)}{2}\right)\\
		& \leq & \NA
		-\frac{(n-k)n^{\alpha}}{n} 
		\frac{(l(\nu)-1)( 2n - (n-k)l(\nu))}{2}\\
		%		& = & \NA
		%		-\frac{(n-k)n^{\alpha}(\frac{n}{n-k}+x-1)}{2n} 
		%		\left( 2n - (n-k)\left(\frac{n}{n-k}+x\right)\right)\\
		& = & \NA
		-\frac{n^{\alpha}}{2n} (n - (1-x)(n-k)) (n-x(n-k))\\
		& = & \NA
		-\frac{n^{\alpha}}{2n} (nk+x(1-x)(n-k)^2)\\
		& \leq & 
		%\NA		-\frac{n^{\alpha}}{2n} nk=
		\NA -\frac{kn^{\alpha}}{2}\,.  
	\end{eqnarray*}	
\end{proof}

\paragraph{}
With $\alpha\geq 1$ the main results (all but Lemma \ref{order}) of Section 2.1 hold with the roles of $\worst$ and $\best$ interchanged,
% $\worst$ replacing the role of $\best$
and so the bound on total variation  in equation \eqref{eqn:UB_eigs} now involves $\eig(\worst)$.   
Therefore, we look to 
bound the eigenvalue of $\worst$ when $\lambda_{1}=n-k$, and to do so we need to 
introduce a new tableau,  $T_{\lambda}^{\searrow}$.

\begin{defn}
	Let $T_{\lambda}^{\searrow}$ define the Young tableau formed by 
	filling in the diagonals of $\lambda$ from left to right, with each 
	diagonal filled from top to bottom. For example,
				\[T_{(3,2,1)}^{\searrow} = 
		\ytableausetup{mathmode,baseline,aligntableaux=center,boxsize=1.2em} \begin{ytableau} 3 & 5 & 6 \\ 2 & 4  \\ 1
		\end{ytableau} \, ,  \hspace{2cm}  T_{(4,2)}^{\searrow} = \begin{ytableau} 
		2& 4& 5& 6\\
		1& 3  	\end{ytableau}\, .\]
\end{defn}

\begin{lemma}
	\label{lem:eigbound}
	For $\alpha \geq1$, and  $\lambda \vdash n$ with $\lambda_{1} =n-k$ we 
	have, 
	\[\eig(T_{\lambda}^{\downarrow}) \leq 
	\eig\left(T_{(n-k,\star)}^{\searrow}\right) .\]
\end{lemma}
\begin{proof}
	Recall that, for any tableau $T_\lambda$, $T_\lambda(m) = j-i+1$ if and only if $m$ appears in box $(i,j)$ of $T_\lambda$. Note that this function is constant on integers appearing in the same diagonal of $T_\lambda$, and that $T_{\lambda}^{\searrow}(m)$ is non-decreasing in $m$.
	
	  Now consider all the values of $\worst(m)$ for $m \in [n]$, including repeats, and order them from smallest to largest as $c_1 \le c_2 \le \dots \le c_n$. Using Lemma \ref{biasedCompute}, we may upper bound $\eig(T_{\lambda}^{\downarrow})$ as follows: 	
	\begin{align*}
	 \NA \cdot \eig(\worst) &= \sum_{m=1}^{n} \worst(m) \cdot 
	m^{\alpha-1} \leq \sum_{m=1}^{n} c_{m} \cdot 
	m^{\alpha-1}\\ 
	&= \sum_{m=1}^{n} T_{\lambda}^{\searrow}(m) \cdot 
	m^{\alpha-1} 
	\leq \sum_{m=1}^{n} T_{(n-k,\star)}^{\searrow}(m) 
	\cdot 
	m^{\alpha-1}  =  \NA\cdot \eig(T_{(n-k,\star)}^{\searrow}) \,.
	\end{align*}
	The first inequality follows from the fact that $m^{\alpha-1}$ is increasing in $m$ for $\alpha \geq 1$, and so pairing up the constants $c_{m}$ and $m^{\alpha-1}$ cannot decrease the value of the sum. 
	For the second inequality, notice that $(n-k,\star)$ is obtained from $\lambda$ by moving boxes up and to the right. Thus the diagonal containing $m$ in $T_{(n-k,\star)}^{\searrow}$ is (weakly) to the right of the corresponding diagonal in $T_{\lambda}^{\searrow}$, and so $T_{\lambda}^{\searrow}(m) \leq T_{(n-k,\star)}^{\searrow}(m)$ for all $m$.

\end{proof}

\begin{lemma}
	\label{lem:boxindex}
	For all $k\leq n-2$ and all $m\in[n]$,
	\begin{align}
	T_{(n-k,\star)}^{\searrow} (m) \leq \frac{n-k}{n}\cdot m. \label{eqn:nkbound}
	\end{align}
	
\end{lemma}

\begin{proof}	
	Let us write $l(n-k,\star) = l^{+} = \lceil \frac{n}{n-k} \rceil$ and 
	$l^{-} = \lfloor \frac{n}{n-k} \rfloor$.  	
	If $m_{1},m_{2}$ belong to 
	the same diagonal then $T_{(n-k,\star)}^{\searrow} (m_{1}) = 
	T_{(n-k,\star)}^{\searrow} (m_{2})$, and hence if \eqref{eqn:nkbound} holds for the smallest $m$ on a diagonal it holds for every entry of that diagonal. Furthermore, the bound trivially holds for any $m$ whose box $(i,j)$ satisfies $j-i+1 \leq0$ (for which the left hand side of \eqref{eqn:nkbound} is non-positive). Combining these two observations, we see that it suffices to prove the bound for those values of $m$ which appear in the first row of $T_{(n-k,\star)}^{\searrow}$.
	
	Note that no diagonal can contain more than $l^{+}$ boxes: call diagonals with $l^{-}$ or fewer boxes \emph{short} diagonals, and all others \emph{long} diagonals. Note that long diagonals can only exist when $l^+=l^-+1$. Any long diagonals clearly occur strictly before the short ones, when working left to right along the first row. If the box $(1,j)$ lies on a long diagonal, then the numbering pattern for $T^{\searrow}$ implies that this box will contain the integer $m=\binom{l^+}{2} + 1 + (j-1)l^+$. For this value of $m$, the right hand side of \eqref{eqn:nkbound} becomes
	\begin{align*}
	\frac{n-k}{n} \left( \frac{l^+(l^+-1)}{2} + 1 + (j-1)l^+ \right)\, & = \,
	\frac{(n-k)l^+}{n} \left( \frac{l^+-1}{2} + \frac{1}{l^{+}}-1 + j \right) \\
	\,&\ge \, j \, = \, T_{(n-k,\star)}^{\searrow} (m) \,,
	\end{align*}
	thanks to the definition of $l^+$ and the fact that $(x-1)/2 + 1/x  \geq 1$ if  $x\geq2$. 
	
	It remains to deal with the short diagonals which contain a box in the first row. For these diagonals we now work from right to left, and consider boxes $(1,n-k+1-j)$ for $j=1,2,\dots$.
	Let $m(j)$ denote the integer appearing in box $(1,n-k+1-j)$. We know that $m(1) = n$, and it is clear that \eqref{eqn:nkbound} holds for box $(1,n-k)$. It is then straightforward to see that 
	\[ m(j) = m(j-1) - \min\{j,l^-\} \]
	for any $j\ge 2$ for which box 	$(1,n-k+1-j)$ is part of a short diagonal. The result for all short diagonals now follows quickly by induction (using $l^-(n-k)/n\le 1$).	

\end{proof}

\begin{lemma}
	\label{lem:genbound2}
	Let $\lambda \vdash n$ with $\lambda_{1} = n-k$. Then the eigenvalue 
	$\eig(\worst)$ for the shuffle $\BLR_{n,\alpha}$ with $\alpha \geq1$ may be 
	bounded as follows:
	\[
	\eig(\worst) \,\le\, 
	\begin{cases}
	1 -	\frac{k(n-k) N_{\alpha-1}(n)}{n \NA} & \quad \text{if $k \le n/4$} \\
	1 - \frac{k}{n} & \quad \text{for all $k$.}  \\
	\end{cases}
	\]
\end{lemma}

\begin{proof}
	We begin by quickly proving the second bound for all $k$ using Lemmas \ref{lem:eigbound} and \ref{lem:boxindex}:
	\[
	\eig(T_{\lambda}^{\downarrow}) \leq 
	\eig\left(T_{(n-k,\star)}^{\searrow}\right)  = \frac{1}{N_\alpha(n)} \sum_{m=1}^n T_{(n-k,\star)}^{\searrow} (m) \cdot m^{\alpha-1}\leq 
	 \frac{1}{N_\alpha(n)} \sum_{m=1}^n \frac{(n-k)}{n} \cdot m^{\alpha}
	=1 
	-\frac{k}{n}\,.
	\]	
	Now we prove the bound for $k\leq n/4$, by finding a tighter bound on 
	$T_{(n-k,\star)}^{\searrow}$ and once again using Lemma \ref{lem:eigbound}. For $k\le n/4$, we know 
	$(n-k,\star) = (n-k,k)$, and thanks to the order in which the boxes are filled, we see that $T_{(n-k,\star)}^{\searrow}(m) \leq m/2$ for $1\leq m\leq 2k$, and $T_{(n-k,\star)}^{\searrow}(m) = m-k$ for $2k+1\leq m\leq n$. This gives us the following simple bound:
	\begin{eqnarray}
	\NA \cdot \eig(T_{(n-k,\star)}^{\searrow}) & \leq & \sum_{m=1}^{2k} 
	\frac{m}{2} m^{\alpha-1} + \sum_{m=2k+1}^{n} (m-k) m^{\alpha-1} \nonumber \\ 
%	& \leq & \NA - \frac{1}{2}\sum_{m=1}^{2k-1} m^{\alpha} -k\sum_{m=2k}^{n} 
%	m^{\alpha-1}.
	& = & \NA - k N_{\alpha-1}(n)+ \left( k N_{\alpha-1}(2k)-  \frac12 N_\alpha(2k)\right) \,. \label{eqn:pre-int}
	\end{eqnarray}
	Now, 
	\begin{eqnarray*}
	 k N_{\alpha-1}(2k)-  \frac12 N_\alpha(2k) & = & k \sum_{m=1}^{2k} m 
	 ^{\alpha-1}\left(1 - \frac{m}{2k}\right) 
	 \leq k \int_{0}^{2k} x^{\alpha-1}\left(1 - 
	 \frac{x}{2k}\right)\mathrm{d} x  \\
	 &=& \frac{k(2k)^\alpha}{\alpha(1+\alpha)} \\
	 &\leq & \frac{2k^2 (n/2)^{\alpha-1}}{\alpha(1+\alpha)} \qquad\text{(since $k\leq n/4$)} \\
	 &\leq& k^2(n/2)^{\alpha-1}\,, 
	\end{eqnarray*}
since $\alpha \ge 1$. Finally, an application of Jensen's inequality shows that $ (n/2)^{\alpha-1} \le N_{\alpha-1}(n)/n$, and combining this with \eqref{eqn:pre-int} yields the desired result.	

\end{proof}

%Note that the above bound for $k\leq n/4$ is slightly off and is discontinuous 
%with the bound present for $\alpha=1$ in Lemma \ref{lem:genbound1}. The 
%accuracy lost here helps us establish one bound for all $\alpha \geq 1$. 

Returning to Theorem \ref{thm:biasedbounds}, recall that $t_{n,\alpha}=\NA/n^{\alpha}$ if 
$\alpha \leq 1$ and $t_{n,\alpha}=\NA/N_{\alpha-1}(n)$ if $\alpha\geq 1$.
Following the argument of Section \ref{UPPERBOUNDANALYSIS} 
and using Lemmas 
\ref{lem:genbound1} and \ref{lem:genbound2} for the appropriate values of $\alpha$, we may show that for any $\alpha\in\mathbb{R}$:
\begin{eqnarray*}
	4\lVert \BLR_{n,\alpha}^{t} - \pi_{n} \rVert^{2}_{\textnormal{TV}} 
	%	& \leq & \left(\eig(1^{n})\right)^{2t} + 2 \sum_{k=1}^{n-2} 
	%	\sum_{\substack{\lambda\,:\,\eig(\worst) \geq 0\\ \lambda_{1} = n-k \\ 
	%			\lambda \neq (n)}} 
	%	d_{\lambda}^{2}  \,	\eig(\worst)^{2t}\\
	& \leq & 
	\left(\eig(1^{n})\right)^{2t} + 2\sum_{k=1}^{n/4} 
	{n\choose k}^{2} k!
	\left(1 - 
	\frac{(n-k)k}{n}t_{n,\alpha}^{-1}\right)^{2t} + 
	2\sum_{k=n/4}^{n-2} {n\choose k}^{2} k!
	\left( 1 - 
	\frac{k}{2}t_{n,\alpha}^{-1}\right)^{2t}.
\end{eqnarray*}
Substituting $t=t_{n,\alpha}(\log n + c)$ and once again 
using the inequality $1-x\le e^{-x}$, we are left with two sums to control;
both have previously been shown to be bounded above by $Ae^{-4c}$, for some 
universal constant $A$, when $c > 5/2$ and $n$ is sufficiently large \cite[page 
42]{Diaconis1988}.

\subsection{Lower Bound for Biased One-sided Transpositions}
\subsubsection{Case: $\alpha\leq 1$}
We use a coupon-collecting argument as in Section~\ref{LOWERBOUND}, once again letting $V_n = 
\{n-n/m+1,\dots,n-1,n\}$ with $m =\log n$, and considering the set $B_n = \{\rho\in S_n \, | \, \text{$\rho$ has at least 1 fixed point in $V_n$}\}$. Equation~\eqref{eqn:indicator_fns} still holds for the biased version of the shuffle, but we now modify the bounds in \eqref{eqn:bound_minus2} and \eqref{eqn:bound_minus1} as follows, using the inequality 
$k^\alpha \le (\frac{m}{m-1})^{1-\alpha}n^\alpha$ for all $k\in V_n$ (which holds for all $\alpha\le 1$):
\begin{align*}
\mathbb{P}(L^{t+1}\in \hat U_n^t) &\,=\, \sum_{k\in \hat 
	U_n^t} \frac{w(k)}{\NA}\frac{k-(n-n/m)}{k} +  \sum_{k\in 
	V_n\setminus \hat U_n^t}\frac{w(k)}{\NA}\frac{|U_n^t|}{k} \nonumber \\
&\,\le\, \frac{ (\frac{m}{m-1})^{1-\alpha}n^\alpha}{\NA}\frac{\sum_{k=1}^{|U_n^t|} k + (n/m -|U_n^t|)|U_n^t|}{(n-n/m)} 
\,\le\,\frac{(\frac{m}{m-1})^{1-\alpha}n^\alpha|U_n^t|}{\NA(m-1)}\,;% \label{eqn:bound_minus2biased} 
\\
\mathbb{P}(R^{t+1}\in U_n^t) &\,\le\, \frac{(\frac{m}{m-1})^{1-\alpha}n^\alpha |U_n^t|}{\NA}\,. %\label{eqn:bound_minus1biased}
\end{align*}
Using these as before we construct a counting process $M_n^t$ which dominates the number of collected cards; the expression for $p_i$ in \eqref{eqn:M_up} becomes 
\[ p_i = \left(\frac{m}{m-1}\right)^{2-\alpha}\frac{n^\alpha}{\NA} \left( \frac nm-i\right) \,, \]
and this is easily checked to be strictly less than one for $n$ sufficiently large, whatever the value of $\alpha\le 1$.
The remainder of the analysis mirrors the unbiased case: using the new expression for $p_i$ the distribution of the random variable $\mathcal T'_i$ is exactly as given in \eqref{eqn:Tprime}, and we arrive at 
\begin{eqnarray*}
	\E[\mathcal{T}^{\prime}] & = & 
	\sum_{i=0}^{n/m-1} 
	\frac{m-1}{mp_i} \ge \frac{\NA}{n^\alpha}(\log n  -\log \log n -(3-\alpha)) \,;\\
	\Var[\mathcal{T}^{\prime}] & \le & \sum_{i=0}^{n/m-1}  \frac{1}{p_i^2}\le
		\frac{\pi^{2}}{6}\frac{\NA^2}{n^{2\alpha}}\,.
\end{eqnarray*}
The proof of the lower bound is completed by using these new bounds in Chebychev's inequality.

\subsubsection{Case: $\alpha \geq 1$}
With $\alpha\ge 1$, the right hand is now more likely to choose cards near the top of the deck, and so it makes sense to swap the roles of the right and left hands in our coupon-collecting argument. To that end, let $V_{n} = \{1,\dots,n/m\}$ and let
\[B_n = \{\rho\in S_n \, | \, \text{$\rho$ has at least 1 fixed point in 
$V_n$}\}\,.\]
As in Section~\ref{LOWERBOUND} we let $U_n^t$ denote the set of 
uncollected cards in $V_n$ after $t$ steps of the biased one-sided transposition shuffle. The change in the number of collected cards in \eqref{eqn:indicator_fns} then holds with $L^{t+1}$ and $R^{t+1}$ interchanged, and we may replace the inequalities in \eqref{eqn:bound_minus2} and \eqref{eqn:bound_minus1} with the following bounds (which hold for sufficiently large $n$):

\begin{align*}
\mathbb{P}(R^{t+1}\in U_n^t) &= \frac{1}{\NA} \sum_{i\in U_n^t} i^{\alpha} \leq \frac{|U_n^t|}{\NA}(n/m)^\alpha
 \leq \frac{|U_n^t|}{\NA} 
\frac{n^{\alpha}}{\alpha (m-1)} \leq \frac{|U_n^t|}{\NA}\frac{ N_{\alpha-1}(n)}{(m-1)}\,; 
%\label{eqn:bound_minusbiga1}
 \\
\mathbb{P}(R^{t+1}\notin U_n^t,L^{t+1}\in U_n^t) &\le \frac{|U_n^t|}{\NA} 
\sum_{i=|U_n^t|+1}^{n} \frac{i^{\alpha}}{i} \leq 
\frac{|U_n^t|N_{\alpha-1}(n)}{\NA}\,.
%\label{eqn:bound_minusbiga2}
\end{align*}
Analysis of the resulting counting process $M_n^t$ shows that for $\alpha\ge 1$ the biased one-sided transposition shuffle satisfies the lowed bound in Theorem~\ref{thm:biasedbounds}.

\section*{Appendix: Lifting Eigenvectors Analysis}
\label{AppendA}

We work with the 
group algebra $\mathfrak{S}_{n} = \mathbb{C}[S_{n}]$ and its representations.
We begin with some of the background and basic constructions.
%for modules of $\mathfrak{S}_{n}$ which we shall be studying.

For each $n\in \mathbb{N}$ let $[n] = \{1,\ldots,n\}$ denote the set consisting 
of the first $n$ natural numbers. 
Given $n\in \mathbb{N}$ we denote by $W^n$ the set of words of length $n$ in 
the elements of $[n]$, 
where by a word of length $n$ we simply mean a string $w=w_1\cdot w_2\cdot \ldots \cdot
w_n$ of $n$ elements from $[n]$, allowing repeats.
Note that in forming words we simply regard the elements of $[n]$ as distinct 
symbols; we separate symbols by a dot $\cdot$. 
(It is notationally convenient later on that these symbols are positive 
integers.)
There is a natural action of the symmetric group $S_n$ on $W^n$: given a word 
$w=w_1\cdot w_2\cdot \ldots \cdot w_n \in W^n$ and an element $\sigma \in S_n$, we let 
$\sigma(w) := w_{\sigma^{-1}(1)}\cdot w_{\sigma^{-1}(2)}\cdot \ldots \cdot 
w_{\sigma^{-1}(n)} \in W^n$. 
We emphasise that this is the action of $S_n$ on words by place permutations, 
it is NOT the action of $S_n$ induced by its action on $[n]$, 
e.g. $(123) (2\cdot3\cdot2) = 2\cdot2\cdot3  \neq 3\cdot1\cdot3$.
We denote by $M^n$ the complex space with basis $W^n$, so $M^n$ is an 
$n^n$-dimensional vector space over $\mathbb{C}$.
The $S_n$-action by place permutations on $W^n$ extends to give $M^n$ the 
structure of an $\mathfrak{S}_n$-module.

To each word $w \in W^n$ we can associate an $n$-tuple of non-negative integers, which we call its \emph{evaluation}, as follows. For $1\leq i \leq n$, let $\eval_i(w)$ count the number of 
occurrences of the symbol $i$ in the word $w$, and then let $\eval(w):= 
(\eval_1(w),\ldots,\eval_n(w))$.
Note that $\sum_{i=1}^n \eval_i(w) = n$ for any word $w$ in $M^n$.
If in addition $\eval(w)$ is a non-increasing sequence of integers, then we 
identify $\eval(w)$ with the corresponding partition of $n$ (this is a matter 
of ``forgetting'' any zeroes at the end of the tuple).

For a partition $\lambda\vdash n$, a Young tableau $T$ of shape $\lambda$ 
naturally corresponds to a word in $W^n$. 
Write the word $w(T) = w_1\cdot w_2\cdot \ldots \cdot w_n$ by setting $w_{T(i,j)} = i$ for 
each box $(i,j)$ in $T$.
Equivalently, the numerical entries in the $i^{\rm th}$ row of $T$ tell us in 
which positions to put the symbol $i$ in the word $w(T)$.
Note that two tableaux give the same word if and only if they have the same 
shape and the same set of entries in each row (i.e., they correspond to the 
same \emph{tabloid}).

\begin{defn}
	To every partition $\lambda \vdash n$ we may 
	associate a simple module  $S^{\lambda}$ of $\mathfrak{S}_{n}$ called the 
	\emph{Specht module} for $\lambda$. The Specht module $S^{\lambda}$ has 
	dimension $d_{\lambda}$.
\end{defn}

\begin{defn}
	For an $n$-tuple $\lambda = (\lambda_1,\ldots,\lambda_n)$ of non-negative 
	integers summing to $n$, define $M^{\lambda}$ to be the span of the words 
	$w \in W^n$ with $\eval(w) = \lambda$. 
	Since $S_n$ acts by place permutations, this is clearly an 
	$\mathfrak{S}_n$-submodule of $M^n$.
	
	If $\lambda\vdash n$ is a partition of $n$, then we allow ourselves an 
	abuse of notation and also consider $\lambda$ as an $n$-tuple by adding 
	some zeroes on the end (if necessary). We can then attach the module 
	$M^\lambda$ to a partition $\lambda$. There is a unique copy of the Specht 
	module $S^{\lambda}$ as a submodule of $M^{\lambda}$
\end{defn}

Let us explicitly identify the unique copy of $S^{\lambda} \subset M^{\lambda}$ 
as mentioned in the above definition.
Let $T$ be a standard Young tableau of shape $\lambda$.
Define $C_T\subseteq S_n$ to be the column stabilizer for $T$ -- that is $C_T$ 
is the subgroup of $S_n$ consisting
of permutations which permute the elements in each column of $T$.
Corresponding to $T$ we have the word $w(T)\in M^\lambda$ as above.
Form a new element 
$s(T)$ of $M^{\lambda}$, where
\[s(T) := \sum_{\sigma \in \textnormal{ColStab}(T)} 
\textnormal{sign}(\sigma)w(\sigma(T)).\]
The Specht module $S^\lambda$ is the subspace of $M^\lambda$ with basis the 
elements $s(T)$, where $T$ runs over all standard 
Young tableaux of shape $\lambda$.

For example, if $\lambda = (3,1)$ then we have $3$ standard Young tableaux,
\[\young(123,4)\hspace{1.5em} \young(124,3) \hspace{1.5em} \young(134,2)\]
and we only have the first column to permute in each case. It is easy to check 
that 
$$
S^{(3,1)} = \langle 
1\cdot1\cdot1\cdot2 - 2\cdot1\cdot 1\cdot1, \;
1\cdot1\cdot2\cdot1-2\cdot1\cdot1\cdot1,\; 
1\cdot2\cdot1\cdot1-2\cdot1\cdot1\cdot1 \rangle.
$$
It is a standard result in the theory that $S^\lambda$ has multiplicity $1$ as a summand of the permutation module $M^\lambda$.

\begin{lemma}[Theorem 2.11.2 of \cite{sagan2013symmetric}]
	\label{lem:youngrule}
	For $\lambda \vdash n$ we have,
	\[M^{\lambda} \cong \bigoplus_{\mu \,\trianglerighteq  \,\lambda} K_{\lambda,\mu} S^{\mu}.\]
	where $K_{\lambda,\mu}S^\mu$ denotes a direct 
	sum of $K_{\lambda,\mu}$ copies of $S^\mu$.
	The coefficients $K_{\lambda,\mu}$ are called Kostka numbers, and for all $\lambda \vdash n$ we know $K_{\lambda,\lambda} = 1$.
\end{lemma}

\paragraph{}
The submodule $M^{(1^{n})}$ is of particular importance because it is spanned 
by the $n!$ permutations of the word $1\cdot 2\cdot \ldots \cdot n \in W^n$, and therefore 
we can model our shuffle on $n$ cards by considering a linear operator on this 
space.
In representation-theoretic terms, $M^{(1^n)}$ can be identified as the regular 
module for $\mathfrak{S}_{n}$;  we may state a classical result about this 
regular module in the notation we have now set up.

\begin{corollary}[Example 2.11.6 of \cite{sagan2013symmetric}]
	\label{classiciso}
	\[M^{(1^{n})} \cong \bigoplus_{\lambda \vdash n} d_{\lambda} S^{\lambda}  \text{ as $\mathfrak{S}_{n}$-modules}.\]
\end{corollary}

For modelling our shuffle $\LR_{n}$ acting on the space $\mathfrak{S}_{n}$ we 
need to turn it into a linear operator. In fact we may turn it into an element 
$Q_{n}$ of our group algebra $\mathfrak{S}_{n}$.

\begin{defn}
	\label{RLop}
	Let $n \in \mathbb{N}$.
	The one-sided transposition shuffle on $n$ cards may be viewed as the 
	following element of the group algebra $\mathfrak{S}_n$.
	\[\sum_{1 \leq i \leq j \leq n}\RL_{n}((i\,j))  (i\,j)= \sum_{1 \leq i 
		\leq j 
		\leq n}\frac{1}{nj} (i\,j)  .\]
	To simplify our calculations it is convenient to scale this operator by 
	$n$, so we introduce a new operator 
	\[Q_n:= \sum_{1 \leq i \leq j \leq n}\frac{1}{j} (i\,j)  .\]
\end{defn}

By realising the shuffling operator as an element of the group algebra we can 
concentrate on finding the eigenvalues of $Q_{n}$ on $M^{(1^{n})}$. 
Furthermore, applying 
Corollary~\ref{classiciso}, we can reduce the problem of finding 
eigenvalues for the shuffle on $M^{(1^n)}$ to the problem of finding 
eigenvalues on the Specht modules $S^\lambda$.
Moreover, since the operator is acting as an element of the group algebra, we 
are then free to study its action on the natural copy of $S^\lambda$ inside 
$M^\lambda$ 
to solve this problem, rather than having to stick to the copies of $S^\lambda$ 
inside $M^{(1^n)}$.
This turns out to be very useful, because there is a natural way of relating 
eigenvectors and eigenvalues corresponding to different partitions
according to the branching rules in Figure \ref{lattice}. Note that we allow the
special case $n=0$ with partition $(0)$ and corresponding Young 
diagram $\emptyset$.

\begin{figure}[H]
	\[
\begin{tikzpicture}
\node (realstart) at (0,-0.9) {$\emptyset$};

\node (start) at (0,0) {\tiny\yng(1)};

\node (2) at (-1,1.5) {\tiny\yng(2)};
\node (11) at (1,1.5) {\tiny\yng(1,1)};

\node (3) at (-2,3) {\tiny\yng(3)};
\node (21) at (0,3) {\tiny\yng(2,1)};
\node (111)  at (2,3)  {\tiny\yng(1,1,1)};

\node (4) at (-3,4.5){\tiny\yng(4)};
\node (31) at(-1.5,4.5) {\tiny\yng(3,1)};
\node (22) at(0,4.5) {\tiny\yng(2,2)};
\node (211) at (1.5,4.5) {\tiny\yng(2,1,1)};
\node(1111) at (3,4.5) {\tiny\yng(1,1,1,1)};

\draw[->] (realstart) -- (start);
\draw[->] (start) -- (2);
\draw[->] (start) -- (11);

\draw[->] (2) -- (21);
\draw[->] (11) -- (111);
\draw[->] (11) -- (21);
\draw[->] (2) -- (3);

\draw[->] (111) -- (211);
\draw[->] (21) -- (211);
\draw[->] (111) -- (1111);
\draw[->] (21) -- (22);
\draw[->] (21) -- (31);
\draw[->] (3) -- (31);
\draw[->] (3) -- (4);

\end{tikzpicture}
\]
	\caption{Young's lattice for partitions of size $n \in \{0,1,2,3,4\}$.}
	\label{lattice}
\end{figure}

\paragraph{}
The one-sided transposition shuffle admits a recursive structure which is seen when we focus on the difference of $Q_{n+1}$ and $Q_{n}$,
\begin{eqnarray}
Q_{n+1} - Q_{n} = \frac{1}{(n+1)}\sum_{1\leq i \leq (n+1)} (i 
\hspace{0.2cm} n+1). \label{eqn:lrdifference}
\end{eqnarray}
This signifies that the only difference between the one-sided transposition shuffle on $n+1$ and $n$ cards is the movement of the new card in position $n+1$. We now define some important linear operators which will allow us to study the 
$Q_n$ inductively using equation \eqref{eqn:lrdifference}.

\begin{defn}
	\label{lin}
	We define two linear operators on the spaces spanned by words. To do so, it 
	is enough to define the effect on any given word. 
	\begin{enumerate}
		\item Let $a \in [n+1]$. Define the \emph{adding 
			operator} $\sh_{a}:M^{n} \rightarrow M^{n+1}$ as follows: given a 
			word $w \in W^n$, define
		\[\sh_{a}(w) := w \cdot a, \]
		(i.e., append the symbol $a$ to the end of the word $w$).
		
		\item Let $a,b \in [n]$. Define the \emph{switching operator} 
		$\Theta_{b,a}:M^{n}\rightarrow 
		M^{n}$ as follows: given a word $w=w_1\cdot w_2\cdot \ldots \cdot w_n \in W^n$, 
		define
		\[\Theta_{b,a}(w) := 
		\sum_{\substack{1 \leq i \leq n \\ w_{i} = b}} 
		w_{1}\cdot \ldots \cdot w_{i-1} 
		\cdot a \cdot w_{i+1}\cdot \ldots \cdot w_{n},\]
		(i.e., for each occurrence of $b$ in the word $w$, replace that 
		occurrence with $a$ and sum the resulting words). 
	\end{enumerate}
\end{defn}

\begin{rem}
	There is some ambiguity in the definitions just given -- since $1\in [n]$ 
	for all $n$, for example, strictly speaking we should define $\sh_1$ 
	separately for each $n$.
	However, this would burden us with even more notation, and it should always 
	be clear from the context which domain and codomain we are considering.
\end{rem}

The adding operator on words defined above is an analogue of the process of 
adding boxes to Young diagrams. The process of adding boxes to Young diagrams 
may be seen in Figure~\ref{lattice}.
We now set up some notation to describe this more precisely. 

\begin{defn}
	Given an $n$-tuple $\lambda=(\lambda_1,\ldots,\lambda_r)$ of non-negative 
	integers summing to $n$ and an element $a \in [n+1]$, we form an 
	$(n+1)$-tuple denoted $\lambda+e_a$ by first adding a zero to the end of 
	$\lambda$ and then adding $1$ to this $(n+1)$-tuple in position $a$. 
	Then $\lambda+e_a$ is an $(n+1)$-tuple of non-negative integers summing to 
	$n+1$.	E.g. $(2,1,0) + e_{3} = (2,1,0,0) + (0,0,1,0) = (2,1,1,0)$.
\end{defn}

With this notation in hand, we have the following easy lemma:
\begin{lemma}
		Given $a \in [n+1]$ and an $n$-tuple $\lambda$ of non-negative integers 
		summing to $n$, we have
		$$
		\sh_{a}: M^{\lambda} \to M^{\lambda+e_a},
		$$
		i.e., the restriction of $\sh_{a}$ to $M^\lambda$ has image in 
		$M^{\lambda+e_a}$. 	
		
		Given $a,b \in [n]$ and $n$-tuples $\lambda,\mu$ of non-negative integers summing to $n$ with $\lambda+e_{a} = \mu +e_{b}$, we have
		$$ \Theta_{b,a}:M^{\lambda} \rightarrow M^{\mu},$$
		i.e., the restriction of $\Theta_{b,a}$ to $M^\lambda$ has image in 
		$M^{\mu}$.
\end{lemma}
Our next result 
establishes the crucial equation upon which all the subsequent results in this 
section rely. It relates the shuffle on $n$ cards to that on $n+1$ cards, and 
allows us to lift the eigenvalues. 

\begin{thm}
	\label{master}
	Given $n \in \mathbb{N}$, we have
	\begin{equation}
	\label{mastereq}
	Q_{n+1} \circ \sh_{a} - \sh_{a} \circ Q_{n} = \frac{1}{n+1}\sh_{a} + 
	\frac{1}{n+1}\sum_{1 \leq b \leq n} \sh_{b}\circ\Theta_{b,a} \,.
%	\tag{$\searrow$}.
	\end{equation}
\end{thm}

\begin{proof}
	It suffices to prove the result on words.
	Let $w=w_1\cdot \ldots \cdot w_n$ be a word of length $n$ and let $a \in [n+1]$.
	Consider the two terms on the left hand side applied to $w$:	
	\begin{eqnarray}
	\label{LH1}
	(Q_{n+1} \circ \sh_{a})(w) = \sum_{1 \leq i \leq j \leq n+1}\frac{1}{j} 
	(i\,j) (w \cdot a) = 
	\frac{1}{n+1}\sum_{\substack{j=n+1 \\ 
			1 \leq i \leq n+1}} (i\,j) (w \cdot a) + 
	\sum_{1 \leq i \leq j \leq n}\frac{1}{j} (i\,j) (w \cdot a)
	\end{eqnarray}
	
	\begin{eqnarray}
	\label{LH2}
	(\sh_{a} \circ Q_{n})(w) = \left(\sum_{1 \leq i \leq j \leq n
	}\frac{1}{j} (i\,j)(w) \right) \cdot a.
	\end{eqnarray}
	
	The second summation in (\ref{LH1}) cancels with (\ref{LH2}) 
	because the adjoined $a$ is in the $(n+1)$-th place, therefore it never 
	moves and may be brought outside.
	This leaves us with the following:
	\begin{eqnarray}
	\label{leftover}
	(Q_{n+1} \circ \sh_{a} - \sh_{a} \circ Q_{n})(w) 
	= \frac{1}{n+1}	\sum_{1 \leq i \leq n+1} (i \hspace{0.2cm} n+1)  (w\cdot a).
	\end{eqnarray}

	If $i=n+1$ we move nothing, giving the term $w\cdot a = \sh_{a}(w)$. 
	Otherwise we apply the transposition 
	$(i \hspace{0.2cm} n+1)$ to $w \cdot a$.
	This has the same effect as replacing the $i^{\rm th}$ symbol $w_i$ in $w$ 
	with $a$ and then appending $w_i$ on the end of the new word. 
	Since we do this for all symbols in $w$, the net effect is the same as 
	$\sum_{1\leq b \leq n} \sh_b \circ \Theta_{b,a}$ applied to $w$. 
	(The operator $\Theta_{b,a}$ systematically 
	finds all occurrences of the letter $b$ in $w$ and replaces with an $a$, 
	and then $\sh_b$ puts the $b$ back on the end. Since $w \in W^n$,
	all possibilities are exhausted by letting $b$ range over $1\leq b \leq 
	n$.)  
	This completes the proof. 
\end{proof}

In terms of shuffling cards, we can interpret (\ref{mastereq}) as taking into 
account the difference between shuffling a deck and then adding a card versus 
adding a card and then shuffling.
If we can understand how the operators $\sh_{a}$ and $\Theta_{b,a}$ behave, 
this then inductively tells us how the shuffle on 
$n+1$ cards behaves using information about the shuffle on $n$ cards. 
Equation (\ref{mastereq}) is our analogue of the similar equation found in 
\cite[Theorem 38]{dieker2018spectral}. We now record a key property of the 
switching operators $\Theta_{a,b}$.

\begin{lemma}[See Section 
	2.9 of \cite{sagan2013symmetric}]\label{thetamorphism}
	The maps $\Theta_{b,a}$ are $\mathfrak{S}_{n}$-module morphisms. 
\end{lemma}	

\begin{proof}
	This is clear from the definitions: since $S_n$ is acting by place 
	permutations, it amounts to the same thing to replace an occurrence of the 
	symbol $b$ with a symbol $a$ and then permute the word as to first permute 
	the word and then replace the same symbol $b$ in its new position with an 
	$a$. 
\end{proof}

The above result is used to prove our next important lemma. 
Recall that given a partition $\lambda$ we may add boxes to it in certain 
places to form a new partition. 
By our blurring of the distinction between partitions of $n$ and $n$-tuples, 
if we add a box on row $i$ the new partition formed is $\lambda+e_{i}$.
Our next lemma shows how our switching operators behave when 
restricted to Specht modules. 

\begin{lemma}[Lemma 44 of \cite{dieker2018spectral}]
	\label{restrict}
	Let $\lambda,\alpha \vdash n$ be such that $\lambda + e_a=\alpha+e_b$ for 
	some $a,b\in[n]$. 
	Then $\Theta_{b,a}$ is non-zero on $S^{\lambda}$ if and only if $\lambda$ 
	dominates $\alpha$. 
	In particular, if $a<b$ then $\Theta_{b,a}(S^{\lambda}) =0$.
\end{lemma}	

\begin{proof}
	Since $S^\lambda$ is simple and $\Theta_{b,a}$ is a module homomorphism, 
	the image  $\Theta_{b,a}(S^\lambda)$ is $0$ or isomorphic to $S^\lambda$, 
	by Schur's lemma.
	But $\Theta_{b,a}(S^\lambda)$ lies in $M^\alpha$ because of the 
	relationship $\lambda+e_a = \alpha+e_b$, and $M^\alpha$ has a submodule 
	isomorphic to $S^\lambda$ if and only if $\lambda$ dominates $\alpha$ (Lemma~\ref{lem:youngrule}). This 
	gives the first assertion of the lemma.
	
	To finish, note that in terms of diagrams the fact that $\lambda+e_a = 
	\alpha+e_b$ corresponds to the fact that we can get from the diagram for 
	$\lambda$ to that for $\alpha$ by moving a box from row $b$ to row $a$.
	Hence, under the given hypothesis, we have that $\lambda$ dominates 
	$\alpha$ if and only if $a\geq b$.
\end{proof}

The preceding result shows that when we restrict equation 
\eqref{mastereq} to a Specht module $S^\lambda$ we can change the index of the 
summation in the final term on the right hand side, as follows.

\begin{corollary}[Similar to Corollary 45 of \cite{dieker2018spectral}]
	\begin{align}
	(Q_{n+1} \circ \sh_{a} - \sh_{a} \circ Q_{n})|_{S^{\lambda}} = 
	\frac{1}{n+1}\sh_{a}|_{S^{\lambda}} + 
	\frac{1}{n+1}\sum_{1 \leq b \leq a} \sh_{b}\circ\Theta_{b,a}|_{S^{\lambda}} \label{eqn:restricted}
	\end{align}
\end{corollary}

Having restricted equation (\ref{mastereq}) to the Specht module 
$S^{\lambda}$, we now analyse the image in the module $M^{\lambda +e_{a}}$
(note that it is clear from the left hand side of \eqref{mastereq} that we do 
land in $M^{\lambda+e_a}$). 

\begin{lemma}[Lemma 41 of \cite{dieker2018spectral}]
	\label{lives}
	Suppose $\lambda\vdash n$ and $\lambda+e_a\vdash n+1$.
	Then the subspace $\sh_{a}(S^{\lambda})$ is contained in an 
	$\mathfrak{S}_{n+1}$-submodule of $M^{\lambda+e_{a}}$ that is isomorphic to 
	$\oplus_{\mu} S^{\mu}$, where the sum ranges over the partitions $\mu$ obtained 
	from $\lambda$ by adding a box in row $i$ for $i\leq a$.
\end{lemma}

\begin{proof}
	Let $w$ be a word of length $n$, so that $\sh_{a}(w) = w \cdot a$.
	If the symbol $b$ does not occur in $w$ then
	\[\sh_{a}(w) = \Theta_{b,a}(\sh_{b}(w)).\]
	Let $b =l(\lambda) +1$, so $b$ does not appear in any $w \in M^{\lambda}$, and 
	consider the $\mathfrak{S}_{n+1}$-submodule $N$ of 
	$M^{\lambda + e_{b}}$ generated by the elements $x\cdot b$ with $x  \in S^{\lambda}$,
	\[N = \langle x\cdot b : x \in S^{\lambda} \rangle.\]
	%Actually we have that the module generated by $\sh_{b}(S^{\lambda})$ is 
	%exactly $N$.
	%	\texttt{OMR: not needed commentary}
	%	 This is an 
	%	improvement over the operator $\sh_{a}$ used in 
	%\cite{dieker2018spectral}, 
	%	and we shall see the importance of this fact in the next few Lemmas.
	\paragraph{}
	The submodule $N$ is isomorphic to 
	$\text{Ind}^{\mathfrak{S}_{n+1}}_{\mathfrak{S}_{n} \times \mathfrak{S}_{1}} 
	(S^{\lambda} \otimes S^{1})$ (this is essentially the definition of how to induce), and using the branching 
	rules on $S_{n}$ this decomposes as a multiplicity free 
	direct sum of Specht modules $S^{\mu}$, where $\mu \vdash n+1$ and $\lambda 
	\subset \mu$ \cite[Theorem 2.8.3]{sagan2013symmetric}.
	Using the observation at the start of the proof, we obtain
	\[\sh_{a}(S^{\lambda}) = \Theta_{b,a}(\sh_{b}(S^{\lambda})) \subseteq 
	\Theta_{b,a}(\langle \sh_{b}(S^{\lambda}) \rangle) = \Theta_{b,a}(N) 
	\cong \bigoplus_{\mu 
		\supseteq \lambda} \Theta_{b,a}(S^{\mu}).\]
	Now note that $\Theta_{b,a}$ sends any word with evaluation $\lambda+e_b$ to a word with evaluation $\lambda+e_a$, and hence
	$\Theta_{b,a}(M^{\lambda +e_{b}}) \subseteq M^{\lambda +e_{a}}$. 
	It follows that all nonzero summands $S^\mu$ appearing on the right hand side occur for $\mu\vdash n+1$ dominating $\lambda+e_a$, and then
	by Lemma \ref{restrict} we can conclude that $\mu$ is obtained from $\lambda$ by adding a cell in row $i$ with $i \leq a$, as required.
\end{proof}

Recall that for any $\mathfrak{S}_n$-module $V$ and a partition $\lambda\vdash n$, we have the \emph{isotypic projection} 
$\iso^{\lambda}:V \rightarrow V$ which projects onto the $S^\lambda$-isotypic component of $V$. 
Using these projections, we can now define our lifting operators, which will be proven to map eigenvectors of $Q_{n}$ 
to those of $Q_{n+1}$. 

\begin{defn}
	Suppose $\lambda \vdash n$ and $\lambda + e_{a} = \mu \vdash n+1$ are two 
	partitions. 	
	Define the \emph{lifting operator} 
	$$
	\pro_{a}^{\lambda,\mu} := \iso^{\mu} \circ \sh_{a}:S^{\lambda} \rightarrow S^{\mu}\subseteq M^\mu.
	$$ 
	Note that since $\sh_a(S^\lambda)\subseteq M^\mu$ and $M^\mu$ contains a unique copy of $S^\mu$, 
	the image of $S^{\lambda}$ under $\pro_a^{\lambda,\mu}$ is actually contained in $S^\mu$.
\end{defn}

The next results give some properties 
of the lifting operators; 
these properties depend in an essential way on the choice of $\sh_{a}$ above. 
Our choice for $\sh_{a}$ gives 
the eigenspaces a different 
structure to those in \cite{dieker2018spectral}, in particular we are able to find \emph{all} the 
eigenvectors for a module $S^{\mu}$ by looking at lifted eigenvectors from partitions $\lambda \subset \mu$.

\begin{corollary}
	\label{nonzero}
	For any $\lambda \vdash n$ and $\lambda+e_{a} \vdash n+1$, there exists 
	some $v \in S^{\lambda}$ with
	\[\pro^{\lambda,\lambda+e_{a}}_{a}(v) \neq 0. \]
\end{corollary}

\begin{proof}
If $\pro^{\lambda,\lambda+e_{a}}_{a}(S^\lambda) = 0$, then the image $\sh_a(S^\lambda)$ 
lies in the kernel of the projection $\iso^{\lambda+e_a}:M^{\lambda+e_a}\to S^{\lambda+e_a}$,
which is an $\mathfrak{S}_{n+1}$-submodule with no component equal to $S^{\lambda+e_a}$.
Hence the submodule generated by $\sh_a(S^\lambda)$ has no component equal to $S^{\lambda+e_a}$.
But we previously observed that (with notation as in the proof of Lemma \ref{lives})
	\[\langle \sh_{a}(S^{\lambda}) \rangle =  \langle\Theta_{b,a}( 
	\sh_{b}(S^{\lambda}))\rangle=\Theta_{b,a}( 
	\langle\sh_{b}(S^{\lambda})\rangle ) \cong \Theta_{b,a}(N) 
	\cong \bigoplus_{1\leq i \leq a} S^{\lambda + e_{i}}.\]
Since the right hand side contains $S^{\lambda+e_a}$ as a summand, we have a contradiction.
\end{proof}

We already know the map $\iso^{\mu}$ is an $\mathfrak{S}_{n+1}$-module 
morphism. Let us realise $\mathfrak{S}_{n}$ inside $\mathfrak{S}_{n+1}$ as the stabilizer of the $(n+1)^{\rm th}$ position. 
Then any $\mathfrak{S}_{n+1}$-module gives rise to an $\mathfrak{S}_n$-module by restriction.

\begin{lemma}
	\label{modulemorph}
	The linear operator $\pro_{a}^{\lambda,\lambda+e_{a}}$ is a 
	$\mathfrak{S}_{n}$-module morphism with trivial kernel.
	%Furthermore we have 
	%$\ker(\pro_{a}^{\lambda,\lambda+e_{a}}) = \{0\}$. Therefore, the maps  
	%$\pro^{\lambda, \lambda + e_{a}}_{a}$ are injective linear maps. 
\end{lemma}

\begin{proof}
	The key observation is that $\sh_a(\sigma(v)) = \sigma(\sh_a(v))$ for all $v \in S^{\lambda}$ and $\sigma \in 
	\mathfrak{S}_{n} \subset \mathfrak{S}_{n+1}$. 
	This is obvious, since $\sh_a$ adds an element in the final position which is not affected by $\sigma$. 
	Hence $\pro_a^{\lambda,\lambda+e_a}$ is the composition of two $\mathfrak{S}_n$-module morphisms.
The final observation follows from Corollary \ref{nonzero} -- since $\pro_a^{\lambda,\lambda+e_a}$
is a nonzero module morphism with a simple module as its domain, it must be injective by Schur's lemma.	
\end{proof}

That the maps $\pro^{\lambda, \lambda + e_{a}}_{a}$ 
are injective is a key point which simplifies the analysis in this paper compared to that in 
\cite{dieker2018spectral}, where the lifting operators can kill eigenvectors.
The next results 
show that $\pro_{a}^{\lambda,\lambda+e_{a}}$ does indeed lift eigenvectors of 
$Q_{n}$ into those of $Q_{n+1}$. 
To establish this we apply our projection $\iso^{\lambda+e_{a}}$ to equation~(\ref{eqn:restricted}). 
We can now state our versions of \cite[Lemma 
48, Theorem 49]{dieker2018spectral};
the proofs follow \emph{mutatis mutandis} from the ones given there (the 
changes 
needed are to the constants in equation (\ref{mastereq})).

\begin{lemma}[Lemma 48 of \cite{dieker2018spectral}]
	For $\lambda = (\lambda_{1},\dots ,\lambda_{r}) \vdash n$, $a \in 
	\{1,2,\dots r+1\}$ and $\mu = \lambda + e_{i}$ for some $1\leq i \leq a$,
	
	\[	Q_{n+1} \circ \pro_{a}^{\lambda,\mu} - 
	\pro_{a}^{\lambda,\mu} \circ Q_{n} = \frac{2 
		+\lambda_{a} -a}{n+1}\pro_{a}^{\lambda,\mu} + 
	\frac{1}{n+1}\sum_{i \leq b \leq a} \Theta_{b,a} \circ 
	\pro_{b}^{\lambda,\mu}.\]
	
\end{lemma}

\begin{proof}
	This follows from the work in \cite{dieker2018spectral}: because we 
	have not changed the switching operators $\Theta_{b,a}$ the proof still 
	holds. The values on the right hand side change to reflect the difference 
	in our equation (\ref{mastereq}).
\end{proof}

\begin{thm}[Theorem 49 of \cite{dieker2018spectral}]
	\label{liftappendix}
	For $\lambda = (\lambda_{1},\dots ,\lambda_{r}) \vdash n$, $a \in 
	\{1,2,\dots r+1\}$ and $\mu = \lambda + e_{i}$ with $1\leq i \leq a$,
	
	\[	Q_{n+1} \circ \pro_{a}^{\lambda,\mu} - 
	\pro_{a}^{\lambda,\mu} \circ Q_{n} = \frac{(2 + 
		\lambda_{i} -i)}{n+1}\pro_{a}^{\lambda,\mu} .\]
	
	In particular if $v \in S^{\lambda}$ is an eigenvector of $Q_{n}$ with 
	eigenvalue $\varepsilon$, then either 
	$\pro_{a}^{\lambda,\mu}(v) = 0$ or 
	$\pro_{a}^{\lambda,\mu}(v)$ is an eigenvector of 
	$Q_{n+1}$ with eigenvalue 
	
	\[\varepsilon + \frac{2 +\lambda_{i} -i }{n+1}.\]
\end{thm}

\begin{proof}
	This proof also follows from the work in \cite{dieker2018spectral}.
\end{proof}

The above theorem tells us exactly how to turn eigenvectors of $Q_{n}$ into 
those of $Q_{n+1}$ and, crucially, it also shows how the eigenvalues change
in value. 
The final part of the analysis rests on showing that \emph{all} of the eigenvectors in a Specht module $S^{\mu}$ can be retrieved
by lifting from Specht modules $S^\lambda$ with $\mu = 
\lambda+e_a$. 
In fact, we show that these lifted eigenvectors form a basis of $S^\mu$.

\begin{thm}
	\label{mapsnew}
	For any $\mu \vdash n+1$ we may find a basis of eigenvectors of $Q_{n+1}$ 
	for the module $S^{\mu}$, 
	by lifting the eigenvectors of $Q_{n}$ belonging in the modules 
	$S^{\lambda}$ with $\lambda\vdash n$ and $\lambda \subset \mu$.
\end{thm}

\begin{proof}
	
	We proceed by induction, for $n=1$ we know that the simple modules 
	$S^{(2)}, S^{(1,1)}$ of $\mathfrak{S}_{n+1}$ 
	are both one dimensional. Therefore, the eigenvector $a \in S^{(1)}$ when 
	lifted indeed forms a basis for each simple module.
	
	Consider the simple module of $S^{\mu}$ with $\mu \vdash n+1$. We know 
	classically from the branching rules of $S_{n}$ \cite[Theorem 
	2.8.3]{sagan2013symmetric} that the restriction 
	of this module to $\mathfrak{S}_{n}$ is 
	given by the 
	\[\textnormal{Res}^{\mathfrak{S}_{n+1}}_{\mathfrak{S}_{n}}(S^{\mu})\cong 
	\bigoplus_{\substack{\lambda
	 \vdash n \\\lambda \subset \mu}} 
	S^{\lambda}.\]
	Now suppose we have a basis of eigenvectors for every $S^{\lambda}$. By Lemma~\ref{modulemorph} the map $\pro^{\lambda,\mu}(S^{\lambda})$ 
	gives a basis for $S^{\lambda}$ inside of the vector space of 
	$\textnormal{Res}^{\mathfrak{S}_{n+1}}_{\mathfrak{S}_{n}}(S^{\mu})$ which 
	is the same vector space as $S^{\mu}$. Hence,
	considering all of the lifted eigenvectors from every $S^{\lambda}$ 
	together we find a basis for $S^{\mu}$. By Theorem~\ref{liftappendix} the 
	lifted eigenvectors form a basis of 
	eigenvectors for $S^{\mu}$.
\end{proof}

Inductively, for any $\lambda\vdash n$, Theorem \ref{mapsnew} gives us the 
 way to find all 
 the eigenvectors for $Q_{n}$ belonging to the Specht 
 module  $S^{\lambda}$: starting at $\emptyset$ and recursively applying 
 lifting 
 operators until we 
 reach $S^\lambda$ will give us an eigenvector, and all eigenvectors arise in 
 this way. 
 Note that $S^{\emptyset}$ has no eigenvectors attached to it, but we allow 
 $\emptyset$ to be an eigenvector with eigenvalue $0$, and $\sh_{a}(\emptyset) 
 =a$. This agrees with the formula in Theorem~\ref{liftappendix} because $a$ is the only 
 eigenvector of $\RL_{1}$ with eigenvalue $1 = 0 + (2+0-1)/(1)$.
 The inductive process of lifting naturally forms one path up Young's 
 lattice which starts at $\emptyset$ and ends at $\lambda$. Furthermore, by Theorem 
 \ref{mapsnew} each unique path we take  $\emptyset \rightarrow \lambda$ will 
 result 
 in a distinct eigenvector for $S^{\lambda}$, and all these eigenvectors together 
 form 
 a basis. We now are in a position to prove Theorem \ref{Maintheorem}.

\begin{proof}[Proof of Theorem \ref{Maintheorem}]
	Every eigenvector in our constructed basis gives a distinct eigenvalue of $S^{\lambda}$, hence there are $d_{\lambda}$ distinct eigenvalues. These are eigenvalues for the shuffle $Q_{n}$, and each one appears $d_{\lambda}$ times due to the isomorphism in Corollary~\ref{classiciso}. 
	Therefore overall we have found $\sum_{\lambda \vdash n} d_{\lambda}^{2} = 
	n!$	eigenvalues and thus have a complete set.
\end{proof}

\begin{proof}[Proof of Lemma~\ref{compute}]
	Given a standard tableau $T$ of size $n$, we build it up following its 
	labelling and keeping 
	track of the changes in eigenvalue given by Theorem \ref{liftappendix}. 
	When box 
	$(i,j)$ is added to $T$ we get a change in eigenvalue of 
	$\frac{2+\lambda_{i}-i}{n+1} = 
	\frac{2 + (j-1) -i}{T(i,j)} = \frac{j-i+1}{T(i,j)}$. After summing these 
	changes for all boxes $(i,j)$ in $T$ we divide by the size of $T$ to 
	normalise the eigenvalue.
\end{proof}

\bibliography{bibliography}
\bibliographystyle{plain}

\end{document}